\documentclass[10pt,a4paper]{article}
\usepackage[utf8]{inputenc}
\usepackage[english]{babel}
\usepackage{amsmath, amscd}
\usepackage{amsfonts}
\usepackage{amssymb}
\usepackage{hyperref}
\usepackage{mathrsfs}

\newcommand{\E}[0]{\mathbf{E}}

\usepackage{amsthm}
\newtheorem{theorem}{Theorem}[section]
\newtheorem{proposition}[theorem]{Proposition}
\newtheorem{corollary}[theorem]{Corollary}
\newtheorem{lemma}[theorem]{Lemma}

\newenvironment{definition}[1][Definition]{\begin{trivlist}
\item[\hskip \labelsep {\bfseries #1}]}{\end{trivlist}}
\newenvironment{example}[1][Example]{\begin{trivlist}
\item[\hskip \labelsep {\bfseries #1}]}{\end{trivlist}}
\newenvironment{remark}[1][Remark]{\begin{trivlist}
\item[\hskip \labelsep {\bfseries #1}]}{\end{trivlist}}

\usepackage[babel]{csquotes}
\usepackage[backend=bibtex]{biblatex}
\usepackage{pgf,tikz}
\usetikzlibrary{positioning,shapes,decorations.markings}
\usetikzlibrary{arrows}
\usetikzlibrary[patterns]
\usetikzlibrary{cd}

\usepackage[sc]{mathpazo} 
\usepackage{microtype} 

\usepackage{titlesec} 
\titleformat{\section}[block]{\large\scshape\centering}{\thesection.}{1em}{} 
\titleformat{\subsection}[runin]
  {\normalfont\large\bfseries}{\thesubsection}{1em}{}
\titleformat{\subsubsection}[runin]
  {\normalfont\normalsize\bfseries}{\thesubsubsection}{1em}{}

\begin{document}

\title{Random local complex dynamics.}

\author{Lorenzo Guerini and Han Peters}

\maketitle

\begin{abstract}
The study of the dynamics of an holomorphic map near a fixed point is a central topic in complex dynamical systems. In this paper we will consider the corresponding random setting: given a probability measure $\nu$ with compact support on the space of germs of holomorphic maps fixing the origin, we study the compositions $f_n\circ\cdots\circ f_1$, where each $f_i$ is chosen independently with probability $\nu$. As in the deterministic case, the stability of the family of the random iterates is mostly determined by the linear part of the germs in the support of the measure. A particularly interesting case occurs when all Lyapunov indices vanish, in which case stability implies simultaneous linearizability of all germs in $supp(\nu)$.
\end{abstract}

\section{Introduction}
\label{section:intr}
An elementary but fundamental result in the theory of local complex dynamical systems is the following:

\medskip

\emph{Let $f: (\mathbb C^m, 0) \rightarrow (\mathbb C^m, 0)$ be a neutral germ, i.e. all eigenvalues of $df(0)$ have norm $1$. Then $\{f^n\}$ is normal in a neighborhood of the origin if and only if $f$ is locally linearizable and $df(0)$ is diagonalizable.}

\medskip

Our goal in this paper is to generalize this statement to the random setting, i.e. when studying compositions of maps that are chosen i.i.d.. Our result is the following:

\begin{theorem}
\label{neutralmeasure}
Suppose that $\nu$ is a neutral probability measure with compact support on $\mathcal O(\mathbb C^m,0)$. Then the origin lies in the random Fatou set if and only if all the germs in $supp(\nu)$ are simultaneously linearizable, and the semigroup of differentials
\[
dS_\nu:=\{df^n_\omega(0)\,|\,\omega\in supp(\nu)^{\mathbb N},\, n\in \mathbb N\},
\]
is conjugate to a sub-semigroup of $U(m)$.
\end{theorem}

As an illustration, consider the case where all maps in the support of $\nu$ are of the form $z \mapsto \lambda z + z^2$, with $|\lambda| = 1$. The fact that all the maps must be simultaneously linearizable implies that $\nu$ is supported at a single point: a germ with a Siegel disk.

\medskip

Let us be more precise about our setting. Instead of considering normality of a family of iterates $\{f^n\}$ we will consider the family $\mathcal F_\omega=\{f^n_\omega\}$. Here $\omega=(f_n)$ is a sequence of germs each chosen independently with a probability $\nu$ and $f^n_\omega:=f_n\circ\dots\circ f_1$.

The investigation on the random dynamics of holomorphic maps began with the work of Fornaess and Sibony. In \cite{FS} they showed that a generic rational map, with attracting cycles,  admits a neighborhood $W$ such that for almost every sequence of functions, chosen i.i.d. with respect to an absolutely continuous probability measure supported in $W$, the Julia set of the family $\mathcal F_\omega$ has zero measure. Furthermore for every $z\in\widehat{\mathbb C}$ the point $z$ belongs almost certainly to the Fatou set of the family $\mathcal F_\omega$. We refer to the paper of Sumi \cite{S2010} for a generalization of this result.

In this paper we will consider the special case in which all functions in the support of a probability measure $\nu$ share a common fixed point. In particular we will be interested in the comparison of the \textit{random Fatou set}, i.e. the points $z$ that lie in the Fatou set of $\mathcal F_\omega$ for almost every $\omega$, and the Fatou set of the semigroup generated by $supp(\nu)$. Our main motivation is to find analogies between this type of random systems and the study of local fixed point theory.

\medskip

The local dynamics of a holomorphic germ $f$ fixing the origin is one of the earliest problems studied in complex dynamics. A rough first description of the local behavior of $f$ is depends on the eigenvalues of $df(0)$. Depending on their absolute value we distinguish between \textit{hyperbolic} germs, where all eigenvalues have norm unequal to $1$, \textit{neutral} germs, where all eigenvalues have norm $1$, and \textit{semi-neutral} germs.

Determining the linearizability of neutral and semi-neutral germs is a subtle problem. In one dimension a complete description of this phenomenon was given in the works of Cremer \cite{Cremer}, Siegel \cite{Sieg}, Brjuno \cite{Brj} and Yoccoz \cite{Y}.

The local dynamics of germs in several complex variables has many analogies with the one dimensional case, but yet many differences. We refer to the survey of Abate \cite{Asurv} for more details. In this setting the presence of resonances constitute an obstacle to linearizability, as studied by Poincar\'e \cite{Po} for attracting germs. The description of attracting germs was completed in dimension $2$ by Latt\'es \cite{Lat} and in arbitrary dimensions independently by Sternberg \cite{Stern1,Stern2} and Rosay--Rudin \cite{RoRu}.

Neutral and semi-neutral germs in several complex variables have also been intensively studied, particularly in the parabolic case, where all eigenvalues of norm $1$ roots of unity. See for instance the works of \'Ecalle \cite{E3}, Hakim \cite{Ha2}, and Abate \cite{A2} for germs tangent to the identity, and Ueda \cite{U1,U2} and Rivi \cite{Riv} for a description of semi-parabolic germs. For the non-parabolic case consider for example the papers of Bracci--Molino \cite{BraMo} and Bracci--Zaitsev \cite{BrazorfAjeje} for neutral germs, and recent papers of Firsova--Lyubich--Radu--Tanase \cite{LRT, FLRT} for semi-neutral germs.

\medskip

In the random setting we may introduce Lyapunov indices $\kappa_1>\dots> \kappa_s$, which play the same role as the (logarithms of absolute values of the) eigenvalues of $df(0)$, see \cite{Os} or \cite{GoMa} for more information on Lyapunov indices. In analogy with the deterministic case, we say that the measure $\nu$ is \textit{attracting} if $\kappa_1<0$, it is \textit{repelling} if $\kappa_1>0$, is \textit{neutral} if the only Lyapunov index is $\kappa_1=0$, and is \textit{semi-neutral} if $\kappa_1=0$ and $\kappa_2<0$.

In the neutral case, the measure $\nu$ plays a marginal role: the origin lies in the random Fatou set if and only if it lies in the Fatou set of the semigroup generated by $supp(\nu)$. Let us present the proof of this fact in the simpler one-dimensional case when $\nu$ is assumed to be supported on a finite set. The general statement follows from Lemma \ref{trappingcor}, whose proof follows same ideas.

\medskip

\noindent {\bf Proof.}
Suppose that the origin lies in the random Fatou set. Then $|f'(0)|=1$ for every $f\in supp(\nu)$, otherwise the sequence $(f^n_\omega)'(0)$ is almost surely unbounded.

Choose a sequence $\omega_0\in supp(\nu)^{\mathbb N}$, such that $\mathcal F_{\omega_0}$ is normal and  given any finite word $(g_1,g_2,\dots,g_n)$, with $g_i\in supp(\nu)$, there exists $N>0$ for which
\[
f_{i+N}=g_i,\qquad 1\le i\le n.
\]
If $T$ denotes the left shift operator, this is equivalent to saying that for every $g\in S_\nu$ there exists $N,n>0$ for which $f^n_{T^N\omega_0}=g$.

Given $\varepsilon>0$ small enough, we may assume that every $f\in supp(\nu)$ is holomorphic and univalent on $\mathbb B_\varepsilon$. Furthermore by Ascoli--Arzel\'a Theorem, there exists $0<\delta_0<\varepsilon$ such that $f^n_{\omega_0}(\mathbb B_{\delta_0})\subset \mathbb B_\varepsilon$, which implies that every $f^n_{\omega_0}$ is univalent on $\mathbb B_{\delta_0}$. By Koebe quarter Theorem, given $g\in S_\nu$ we have
\[
g(\mathbb B_{\delta_0/4})\subset f^n_{T^N\omega_0}\circ f^N_{\omega_0}(\mathbb B_{\delta_0})\subset \mathbb B_\varepsilon.
\]
Weak Montel Theorem implies normality of the semigroup $S_\nu$. Proving that normality of $S_\nu$ implies that the origin lies in the random Fatou set is trivial. \hfill $\square$

\medskip

We note that the semi-group need not be normal in the attracting (or similarly, semi-neutral) case, as we will see in section \ref{section:nonzero}.

The behavior of semi-neutral measures turns out to be substantially more complicated then the other cases, and we are unable to give a satisfying description of when the origin belongs to the random Fatou set. In section \ref{section:generalprob} we do show that it is not possible decide whether the origin lies in the random Fatou set just by looking at the Lyapunov indices or at linearizability. In the two-dimensional setting, the fact that the origin is in the random Fatou set implies the existence of stable manifolds, analogous to their deterministic setting.
\begin{theorem}
Let $\nu$ be a semi-neutral measure on $\mathbb O(\mathbb C^2,0)$ with compact support. If the origin lies in the random Fatou set, then almost surely every limit germ $g=\lim_{k\to\infty}f^{n_k}_\omega$ has rank one and given $z$ sufficiently close to the origin, its stable set $\mathbb W^s_\omega(z)$ is locally a one-dimensional complex manifold.
\end{theorem}

The paper is organized as follows. In section \ref{section:randommatrices} we review background on random matrices, and prove that a neutral measure of linear maps is stable if and only if there is a conjugation to a sub-semigroup of $U(m)$. In section \ref{section:generalprob} we give a more precise formulation of the problem, and treat several examples showing that our assumptions are necessary. We will study normality of the family $\mathcal F_\omega$ in section \ref{section:nonzero} when $\kappa_1\neq 0$, and in section \ref{section:neutral} for neutral measures. Semi-neutral measures will be considered in section \ref{section:semineutral}

Throughout the paper we use the inductive limit topology on $\mathcal O(\mathbb C^m,0)$. Its construction and properties are discussed in the appendix. 

\medskip {\bf Acknowledgement.} We thank L\'aszl\'o Lempert and Jasmin Raissy for several useful discussions.

\section{Products of random matrices}
\label{section:randommatrices}
Let $(\Omega,\mathscr B,\mu)$ be a probability space and $T:\Omega\rightarrow\Omega$ be an ergodic transformation. Given a measurable function $M:\Omega\rightarrow Mat(m,\mathbb C)$, the \textit{(linear) cocycle} defined by $M$ over $T$ is the skew-product transformation
\[
F:\Omega\times \mathbb C^m\rightarrow \Omega\times \mathbb C^m,\qquad (\omega,v)\mapsto (T\omega, M_\omega v).
\]
Observe that $F^n(\omega,v)=(T^n\omega, M^n_\omega v)$, where
\[
M^n_\omega:=M_{T^{n-1}\omega}\cdots M_\omega.
\]

By elementary properties of measurable functions, the function $\omega\mapsto\sup_n\Vert M^n_\omega\Vert$ is also measurable. Therefore the set
\[
\Omega_b=\{\omega\in \Omega\,|\, M^n_\omega \textrm{ is bounded}\},
\]
is measurable and it satisfies $T^{-1}\Omega_b=\Omega_b$. By ergodicity of $T$ we either have $\mu(\Omega_b)=0$ or $\mu(\Omega_b)=1$.

\medskip

\noindent{\bf Problem}\emph{ For which cocycles is the family $\mathcal M_\omega=\{M^n_\omega\}$ almost surely bounded?}

\medskip

A cocycle of particular importance for our work is the following
\begin{definition}[Definition (i.i.d. cocycle)]
Let $\nu$ be a probability measure $Mat(m,\mathbb C)$. Let $\Omega$ be the space of all sequences in $supp(\nu)$, $T$ be the left shift operator and $M:\Omega\rightarrow Mat(m,\mathbb C)$ be the function which returns the first element of the sequence $\omega$.

Let $U_1,\dots,U_n$ be open subsets of $supp(\nu)$, we define the set
\[
\mathcal C_{U_1,\dots U_n}:=\left\{\omega\in\Omega\,|\, M_\omega\in U_1,\dots, M_{T^{n-1}\omega}\in U_n \right\}.
\]
This collection of sets forms a basis for the product topology of $\Omega$. With $\mathscr B$ we will denote the corresponding Borel $\sigma$-algebra and with $\mu$ the measure for which $\mu(C_{U_1,\dots, U_n})=\nu(U_1)\cdots\nu(U_n)$. It is a well known fact that $T$ is ergodic on $(\Omega,\mathscr B,\mu)$.

The function $M$ is continuous and therefore measurable. We notice that $M_{T^{n-1}\omega}$ coincides with the $n$-th element of the initial sequence and that $X_i(\omega):=M_{T^{i-1}\omega}$ is an i.i.d. sequence of random variables with values in $Mat(m,\mathbb C)$, each chosen with probability $\nu$. In this case, $M^n_\omega$ is given by the product of the first $n$ elements of the sequence $\omega$.
\end{definition}

The problem of the iteration of random matrices was first studied by Furstenberg and Kesten in \cite{F}. An important generalization of their result is the multiplicative ergodic Theorem of Oseledec \cite{Os}. The following version of the theorem can be found in \cite{Rue}.
\begin{theorem}[Multiplicative ergodic Theorem]
\label{odeselecthm}
Let $T:\Omega\rightarrow \Omega$ be an ergodic transformation and $M:\Omega\rightarrow Mat(m,\mathbb C)$ be a measurable function such that
\[
\log^+\Vert M_\omega\Vert\in L^1(\Omega,\mu).
\]
There exist numbers $+\infty>\kappa_1>\dots>\kappa_s\ge -\infty$ called \textbf{Lyapunov indices} and natural numbers $\alpha_1,\dots,\alpha_s$ called \textbf{Lyapunov multiplicities} satisfying the equation $\alpha_1+\dots+\alpha_s=m$, such that, for almost every $\omega\in \Omega$, we have
\begin{enumerate}
\item[a.] There exists a Hermitian matrix $\Lambda_\omega$ with eigenvalues $e^{\kappa_1}>\dots>e^{\kappa_s}$, with respective multiplicities $\alpha_i$,  such that
\begin{equation}
\label{matrixconv}
\lim_{n\to\infty}\left((M^n_\omega)^*M^n_\omega\right)^{\frac{1}{2n}}=\Lambda_\omega.
\end{equation}
\item[b.] Suppose that $\mathcal U_1(\omega),\dots,\mathcal U_s(\omega)$ are the eigenspaces of $\Lambda_\omega$. Let $\mathcal V_{s+1}(\omega)=\{0\}$ and $\mathcal V_i(\omega)=\mathcal U_i(\omega)\oplus\cdots\oplus U_s(\omega)$. Then, given $v\in \mathcal V_i(\omega)\setminus \mathcal V_{i+1}(\omega)$, we have
\begin{equation}
\label{mergthm}
\lim_{n\to\infty}n^{-1}\log\Vert M^n_\omega v\Vert=\kappa_i.
\end{equation}
\end{enumerate}
The set $\sigma=\{(\kappa_1,\alpha_1),\dots,(\kappa_s,\alpha_s)\}$ is called the \textbf{Lyapunov spectrum}.
\end{theorem}
\begin{definition}
Suppose that a given cocycle satisfies $\log^+\Vert M_\omega\Vert\in L^1(\Omega,\mu)$ and let $\kappa_1>\dots>\kappa_s$ be the Lyapunov indices. We say that the linear cocycle is
\begin{enumerate}
\item[]\textbf{Attracting} if the maximal Lyapunov index $\kappa_1<0$,
\item[]\textbf{Repelling} if the maximal Lyapunov index $\kappa_1>0$,
\item[]\textbf{Neutral} if the Lyapunov spectrum $\sigma_\mu=\{(0,m)\}$,
\item[]\textbf{Semi-neutral} if the maximal Lyapunov index $\kappa_1=0$ and $\alpha_1\neq m$.
\end{enumerate}
\end{definition}
\begin{remark}
In the work of Furstenberg and Kesten \cite{F} it is proved that, for every cocycle, we may define a number $\kappa_\mu$, called \textit{Lyapunov exponent}, as
\begin{equation}
\label{lyap}
\kappa_\mu=\lim_{n\to\infty}n^{-1} \E\log \Vert M^n_\omega\Vert.
\end{equation}
Furthermore under the assumption that $\log^+\Vert M_\omega\Vert\in L^1(\Omega,\mu)$, for almost every $\omega\in \Omega$ we have
\begin{equation}
\label{Furstthm}
\lim_{n\to\infty}n^{-1}\log \Vert M^n_\omega \Vert=\kappa_\mu.
\end{equation}
\end{remark}
In the same hypothesis  of Theorem \ref{odeselecthm} we can describe the connection between Lyapunov indices and Lyapunov exponent as follows
\begin{lemma}
\label{lyapexpandind}
The Lyapunov exponent $\kappa_\mu$ coincides with the maximal Lyapunov index $\kappa_1$. Furthermore the following equality holds
\[
\alpha_1\kappa_1+\dots+\alpha_s\kappa_s=\E\log|\det(M_\omega)|.
\]
\end{lemma}
\begin{proof}
Let $\kappa_1$ be the maximal Lyapunov index and choose $\omega\in \Omega$ such that \eqref{matrixconv}, \eqref{mergthm} and \eqref{Furstthm} hold. It is not hard to prove that $\kappa_\mu\ge \kappa_1$, thus it remains to prove that $\kappa_\mu\le \kappa_1$. Let $v_{n}$ be a unit vectors such that $\Vert M^n_\omega\Vert=\Vert M^n_\omega v_n\Vert$. We take $n_k$ such that $v_{n_k}\to v$. By \eqref{mergthm}, for some $i=1,\dots,s$, we have
\begin{align*}
\kappa_i&=\lim_{k\to\infty}n_k^{-1}\log\Vert M^{n_k}_\omega v\Vert \\
&\ge\lim_{k\to\infty}n_k^{-1}\log\left(\Vert M^{n_k}_\omega v_{n_k} \Vert -\Vert M^{n_k}_\omega (v_{n_k}-v)\Vert\right)\\
&\ge \lim_{k\to\infty}n_k^{-1}\log\Vert M^{n_k}_\omega\Vert+\lim_{k\to\infty}n_k^{-1}\log\left(1-\Vert v_{n_k}-v\Vert\right)\\
&\ge \kappa_\mu,
\end{align*}
since $\kappa_i\le \kappa_1$, the equality $\kappa_\mu=\kappa_1$ follows.

If we apply $\log|\,\cdot\,|$ to both side of \eqref{matrixconv}, we obtain that
\begin{align*}
\alpha_1\kappa_1+\dots+\alpha_s\kappa_s&=\log|\det(\Lambda_\omega)|\\
&=\lim_{n\to\infty}\frac{1}{n}\log|\det (M^n_\omega)|\\
&=\lim_{n\to\infty}\frac{1}{n}\sum_{i=0}^{n-1}\log|\det (M_{T^i\omega})|.
\end{align*}
By Birkhoff ergodic theorem the last term of this equality converge almost surely to $\E\log|\det(M_\omega)|$. By choosing an appropriate $\omega$ , we obtain the desired equality.
\end{proof}


\begin{corollary}
\label{equivalentdefmeas}
Suppose that $\log^+\Vert M_\omega\Vert\in L^1(\Omega,\mu)$. We have the following list of equivalences:
\begin{enumerate}
\item[a.]the cocycle is attracting if and only if $\kappa_\mu<0$,
\item[b.]the cocycle is repelling if and only if $\kappa_\mu>0$,
\item[c.]the cocycle is neutral if and only if $\kappa_\mu=0$ and $\E\log|\det(M_\omega)|=0$,
\item[d.]the cocycle is semi-neutral if and only if $\kappa_\mu=0$ and $\E\log|\det(M_\omega)|<0$.
\end{enumerate}
\end{corollary}

When the cocycle is attracting or repelling, \eqref{Furstthm} gives an answer to the problem of the boundedness of the family $\mathcal M_\omega=\{M^n_\omega\}$.
\begin{corollary}
If a cocycle is attracting then the family $\mathcal M_\omega$ is almost certainly bounded in $Mat(m,\mathbb C)$. If a cocycle is repelling  $\mathcal M_\omega$ is almost certainly unbounded.
\end{corollary}

\subsection{Neutral measures}\mbox{}

\medskip\noindent
The stability problem is considerably more complicated in the neutral setting. However, for i.i.d. cocycles arising from probability measures with compact support, we can give a precise description of the stable systems in Proposition \ref{vanishingLyap} below. In the two examples that follow we show that no such statement can hold for general neutral cocycles.

Given a probability measure $\nu$ on $Mat(m,\mathbb C)$ with compact support, by Tychonoff Theorem, the set $\Omega$ is compact, thus $\log^+\Vert M_\omega\Vert\in L^1(\Omega,\mu)$.

\begin{definition}
Let $\nu$ be a measure on $Mat(m,\mathbb C)$ with compact support. We say that $\nu$ is attracting (respectively repelling, neutral and semi-neutral) if the corresponding i.i.d. cocycle is attracting (respectively repelling, neutral and semi-neutral).
\end{definition}


%

\begin{lemma}
\label{fullmeasure}
Let $\nu$ be a probability measure on $Mat(m,\mathbb C)$ with compact support. Then the set
\[
\Omega_{a}:=\left\{\omega\in \Omega\,\Big|\,\forall n\in\mathbb N \textrm{ and }\forall\alpha\in \Omega,\,\exists k_j\,:\, M^n_{T^{k_j}\omega}\to M^n_\alpha\right\}
\]
is a full measure subset of $\Omega$.
\end{lemma}
\begin{proof}
Let $n\in\mathbb N$ and choose a sequence $\varepsilon_j\to 0$. We define
\[
\Omega_{n,j}:=\{\omega\in \Omega\,|\,\forall\alpha\in \Omega,\,\exists k\,:\,\Vert M^n_{T^k\omega}-M^n_\alpha\Vert<\varepsilon_j\}.
\]
Is it clear that $\bigcap_{n,j}\Omega_{n,j}=\Omega_a$. Since the sets $\Omega_{n,j}$ are countably many it suffices to prove that, for every $n,j\in\mathbb N$, the set $\Omega_{n,j}$ has full measure.

Suppose that $n$ and $j$ are fixed. Given $\alpha\in \Omega$, using the continuity of $T$ and $M$, we may find an open neighborhood $U_\alpha\ni \alpha$ so that, given $\beta\in U_\alpha$, we have
\[
\Vert M^n_\alpha-M^n_\beta\Vert<\varepsilon_j/2.
\]

By compactness of $supp(\nu)$, the set $\Omega$ is also compact, therefore we may find $\alpha_1,\dots \alpha_N$ such that $\Omega=\cup_i U_{\alpha_i}$. For simplicity we will write $U_i$ for $U_{\alpha_i}$. Notice that $supp(\mu)=\Omega$, therefore all the sets $U_{i}$ have positive measure. Using standard results from ergodic theory and the fact that the $\alpha_i$ are a finite number, we find that
\[
\Omega_{n,j}'=\{\omega\in \Omega\,|\,\forall i=1,\dots N,\,\exists k_i \,:\, T^{k_i}\omega\in U_{i}\}
\]
has full measure in $\Omega$.

Finally, given $\omega\in\Omega_{n,j}'$ and $\alpha\in \Omega$, there exists $U_i$ so that $\alpha\in U_i$ and $k_i$ so that $T^{k_i}\omega\in U_i$. By the definition of the $U_i$-s we obtain that
\[
\Vert M^n_{T^{k_i}\omega}-M^n_\alpha\Vert\le \Vert M^n_{T^{k_i}\omega}-M^n_{\alpha_i}\Vert+\Vert M^n_{\alpha_i}-M^n_\alpha\Vert<\varepsilon_j,
\]
which proves that $\Omega'_{n,j}\subset \Omega_{n,j}$ and thus that $\Omega_{n,j}$ has full measure.
\end{proof}
\begin{remark}
The previous lemma is valid in a much more general context. As a matter of fact it holds, with the same exact proof, for continuous cocycles, i.e. where both $T$ and $M$ are continuous functions, over a compact space.
\end{remark}
\begin{lemma}
\label{compactgroup}
Let $\mathcal G$ be a compact subgroup of $GL(m,\mathbb C)$. Then $\mathcal G$ is conjugated to a subgroup of the standard unitary group $U(m)$.
\end{lemma}
\begin{proof}
Let $\nu_{\mathcal G}$ the Haar measure on $\mathcal G$ normalized in such a way that $\nu_{\mathcal G}(\mathcal G)=1$. We define a Hermitian inner product on $\mathbb C^m$ as
\[
\langle u,v\rangle_\mu=\int_{\mathcal G}\langle Mu,Mv\rangle_E \,d\nu_{\mathcal G}(M),
\]
where $\langle\,\cdot\,,\,\cdot\,\rangle_E$ is the standard Hermitian inner product on $\mathbb C^m$. Thank to the properties of the measure $\nu_{\mathcal G}$, for every $M\in \mathcal G$ and $u,v\in\mathbb C^m$, we have
\[
\langle M u,Mv\rangle_{\mathcal G}=\langle u,v\rangle_{\mathcal G}.
\]
Let $\beta_E=\{e^i_E\}$ and $\beta_{\mathcal G}=\{e^i_{\mathcal G}\}$ be two orthonormal basis of $\langle\,\cdot\,,\,\cdot\,\rangle_E$ and $\langle\,\cdot\,,\,\cdot\,\rangle_{\mathcal G}$ respectively. Let $H\in GL(m,\mathbb C)$ be the matrix such that $He^i_{\mathcal G}=e^i_E$. Given $u,v\in\mathbb C^m$ we can easily prove that
\begin{align*}
\langle u,v\rangle_{\mathcal G}=\langle Hu, Hv\rangle_E.
\end{align*}
Finally given $M\in \mathcal G$, than for every $u,v\in \mathbb C^m$ we have
\begin{align*}
\langle H^{-1}MHu,H^{-1}MHv\rangle_E&=\langle MH^{-1}u, MH^{-1}v\rangle_\mathcal G\\
&=\langle H^{-1}u,H^{-1}v\rangle_\mathcal G\\
&=\langle u,v\rangle_E\end{align*}
which implies that $H^{-1}\mathcal GH$ is a subgroup of $U(m)$.
\end{proof}

We define the set
\[
S_\nu=\{M^n_\omega\,|\,\omega\in\Omega,\, n\in \mathbb N\}.
\]
In this case $S_\nu\subset Mat(m,\mathbb C)$ has a semigroup structure, which is not necessarily true for general neutral cocycles.
\begin{proposition}
\label{vanishingLyap}
let $\nu$ be a neutral measure on $Mat(m,\mathbb C)$ with compact support, and consider the corresponding i.i.d. cocycle. The following are equivalent
\begin{enumerate}
\item the family $\mathcal M_\omega$ is almost certainly bounded,
\item the semigroup $S_\nu$ is relatively compact in $GL(m,\mathbb C)$ and is conjugated to a sub-semigroup of the unitary group $U(m)$.
\end{enumerate}
\end{proposition}

Before proceeding to the proof, we want to discuss why there is no hope that a (similar) proposition holds for a generic neutral cocycle or even for a continuous neutral cocycle over a compact space. We will show this by constructing two cocycles for which $\mathcal M_\omega$ is almost certainly bounded but for which $S_\Omega=\{M^n_\omega\,|\,\omega\in\Omega,\,n\in\mathbb N\}$ is not relatively compact in $GL(m,\mathbb C)$.

\begin{example}
Let $X=[0,1)$ with the Borel $\sigma$-algebra $\mathscr B$ and Lebesgue measure $\lambda$. Let $\theta\in \mathbb R\setminus \mathbb Q$ and define $T:X\rightarrow X$ by $Tx=x+\theta\mod 1$. It is well known that the transformation $T$ is ergodic.
Given $x_1,x_2\in X$ we write
\[
d_X(x_1,x_2)=\inf_{m\in\mathbb Z}|x_1+m-x_2|.
\]

Let $\varphi\in L^1(X,\lambda)$ be a nonnegative and unbounded function. The function $M:X\rightarrow\mathbb R^+$, defined as
\[
\log(M_x)=f(x):=\varphi(x)-\varphi(Tx),
\]
is measurable and defines a neutral cocycle over $T$.

Notice that $M^n_x=e^{\varphi(x)-\varphi(T^nx)}\le e^{\varphi(x)}$. Since $\varphi(x)<\infty$ for almost every $x\in X$, it follows that $M^n_x$ is almost surely a bounded sequence. On the other hand we can find $K_0>0$ so that the set $X_{K_0}=\{x\in X\,|\, \varphi(x)<K_0\}$ has positive measure. Since $T$ is ergodic, for almost every $x\in X$ there exists $n_0$ such that $T^{n_0}x\in X_{K_0}$. We conclude that $M^{n_0}_{x}>e^{\varphi(x)-K_0}$. Since $\varphi$ is unbounded, it follows that $S_X$ is unbounded in $\mathbb R^+$.

\end{example}
\begin{example}
In the previous example the function $f$ may be unbounded, in which case the unboundedness of $S_X$ is not particularly surprising. In this second example we will construct a function $\varphi$ for which $f$ is continuous, and therefore bounded.

Let $\omega_n= (n-1)\theta \mod 1=T^{n-1}0$ and let $a_k=2^kk$. For every $k>0$ we may choose $\varepsilon_k>0$ such that
\begin{itemize}
\item[a.] $\varepsilon_k\le 1/a_k^2$;
\item[b.] the sets $[\omega_1-\varepsilon_k,\omega_1+\varepsilon_k],\dots,[\omega_{2a_k}-\varepsilon_k,\omega_{2a_k}+\varepsilon_k]$ are pairwise disjoint.
\end{itemize}
Write $I_{j,k}=[\omega_j-\varepsilon_k,\omega_j+\varepsilon_k]$ and define $f_k$ as
\[
f_k(x):=\begin{cases}-\frac{1}{2^k}\left(1-\frac{d_X(x,\omega_j)}{\varepsilon_k}\right) &\textrm{if } x\in I_{j,k}\textrm{ for } j=1,\dots, a_k;\\
+\frac{1}{2^k}\left(1-\frac{d_X(x,\omega_j)}{\varepsilon_k}\right) &\textrm{if } x\in I_{j,k}\textrm{ for } j=a_k+1,\dots, 2a_k;\\
0&\textrm{elsewhere}.\end{cases}
\]
We define a corresponding function $\varphi_k$ as
\[
\varphi_k(x):=\begin{cases}-\sum_{i=1}^{j-1}f_k(T^{-i}x)&\textrm{if }x\in I_{j,k}\textrm{ for } j=2,\dots, 2a_k;\\
0 &\textrm{elsewhere}.\end{cases}
\]
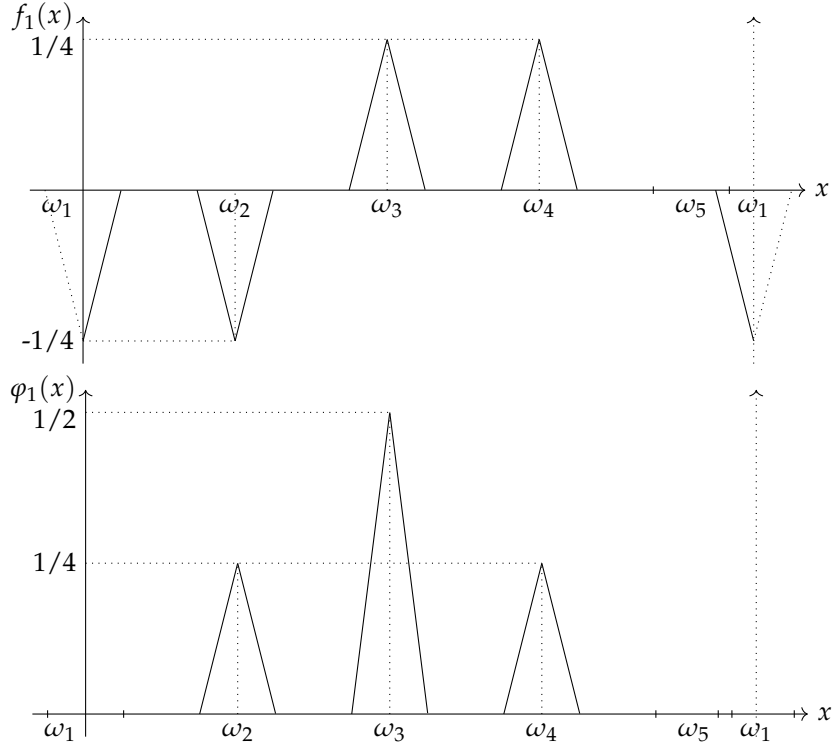
\begin{figure}[ht!]
\centering
\begin{tikzpicture}[domain=-1:9.8]
    \draw[->] (-0.7,0) -- (9.5,0) node [right]{$x$};
    \draw[->] (0,-2.3) -- (0,2.3) node [left] {$f_1(x)$};
    \draw[dotted][->] (8.82,-2.3) -- (8.82,2.3);

    \draw[dotted] plot [domain=-0.5:0,samples=100] (\x,{-2-4*\x});
    \draw plot [domain=0:0.5,samples=100] (\x,{-2+4*\x});

    \draw plot [domain=1.5:2,samples=100] (\x,{-2-4*(\x-2)});
    \draw plot [domain=2:2.5,samples=100] (\x,{-2+4*(\x-2)});
    \draw[dotted] plot [domain=-2:0,samples=100] (2,\x);

    \draw plot [domain=3.5:4,samples=100] (\x,{+2+4*(\x-4)});
    \draw plot [domain=4:4.5,samples=100] (\x,{+2-4*(\x-4)});
    \draw[dotted] plot [domain=0:2,samples=100] (4,\x);

    \draw plot [domain=5.5:6,samples=100] (\x,{+2+4*(\x-6)});
    \draw plot [domain=6:6.5,samples=100] (\x,{+2-4*(\x-6)});
    \draw[dotted] plot [domain=0:2,samples=100] (6,\x);

    \draw plot [domain=-0.05:0.05, samples=100] (7.5,\x);
	\draw plot [domain=-0.05:0.05, samples=100] (8.5,\x);
	
    \draw plot [domain=8.32:8.82,samples=100] (\x,{-2-4*(\x-8.82)});
    \draw[dotted] plot [domain=8.82:9.32,samples=100] (\x,{-2+4*(\x-8.82)});

    \draw[dotted] plot [domain=0:6,samples=100](\x,{2});
    \draw[dotted] plot [domain=0:2,samples=100](\x,{-2});

    \node[below left] at (0,0) {$\omega_1$};
    \node[below] at (2,0) {$\omega_2$};
    \node[below] at (4,0) {$\omega_3$};
    \node[below] at (6,0) {$\omega_4$};
    \node[below] at (8,0) {$\omega_5$};
    \node[below] at (8.82,0) {$\omega_1$};
    \node[left] at (0,1.9) {1/4};
    \node[left] at (0,-2) {-1/4};
\end{tikzpicture}
\begin{tikzpicture}[domain=-1:9.8]
    \draw[->] (-0.7,0) -- (9.5,0) node [right]{$x$};
    \draw[->] (0,-0.3) -- (0,4.3) node [left] {$\varphi_1(x)$};
    \draw[dotted][->] (8.82,-0.3) -- (8.82,4.3);

    \draw plot [domain=-0.05:0.05, samples=100] (-0.5,\x);
	\draw plot [domain=-0.05:0.05, samples=100] (0.5,\x);

    \draw plot [domain=1.5:2,samples=100] (\x,{+2+4*(\x-2)});
    \draw plot [domain=2:2.5,samples=100] (\x,{+2-4*(\x-2)});
    \draw[dotted] plot [domain=0:2,samples=100] (2,\x);

    \draw plot [domain=3.5:4,samples=100] (\x,{+4+8*(\x-4)});
    \draw plot [domain=4:4.5,samples=100] (\x,{+4-8*(\x-4)});
    \draw[dotted] plot [domain=0:4,samples=100] (4,\x);

    \draw plot [domain=5.5:6,samples=100] (\x,{+2+4*(\x-6)});
    \draw plot [domain=6:6.5,samples=100] (\x,{+2-4*(\x-6)});
    \draw[dotted] plot [domain=0:2,samples=100] (6,\x);

    \draw plot [domain=-0.05:0.05, samples=100] (7.5,\x);
	\draw plot [domain=-0.05:0.05, samples=100] (8.5,\x);
	
	\draw plot [domain=-0.05:0.05, samples=100] (8.32,\x);
	\draw plot [domain=-0.05:0.05, samples=100] (9.32,\x);

    \draw[dotted] plot [domain=0:6,samples=100](\x,{2});
    \draw[dotted] plot [domain=0:4,samples=100](\x,{4});

    \node[below left] at (0,0) {$\omega_1$};
    \node[below] at (2,0) {$\omega_2$};
    \node[below] at (4,0) {$\omega_3$};
    \node[below] at (6,0) {$\omega_4$};
    \node[below] at (8,0) {$\omega_5$};
    \node[below] at (8.82,0) {$\omega_1$};
    \node[left] at (0,2) {1/4};
    \node[left] at (0,3.9) {1/2};
\end{tikzpicture}
\caption{the graphs of the functions $f_1$ and $\varphi_1$}
\end{figure}

With some elementary calculations, one can verify that $f_k(x)=\varphi_k(x)-\varphi_k(Tx)$ and that
\[
\int_X\varphi_k d\lambda=a_k^2\varepsilon_k\frac{1}{2^k}\le \frac{1}{2^k}.
\]

Let $f(x)=\sum_{k=1}^\infty f_k(x)$ and $\varphi(x)=\sum_{k=1}^\infty \varphi_k(x)$. Since $|f_k(x)|\le (1/2)^k$, the function $f(x)$ is continuous. On the other hand, since all the $\varphi_k$ are nonnegative functions, also $\varphi$ is nonnegative. Furthermore the above estimate on the integral of $\varphi_k$ implies that $\varphi\in L^1(X,\lambda)$. Finally we notice that $\varphi_k(\omega_{a_k+1})=k$, and therefore that $\varphi$ is unbounded. The function $\varphi$ we constructed satisfies all the hypotheses of the previous example and $f(x)=\varphi(x)-\varphi(Tx)$ is a continuous function.

\end{example}

\begin{proof}[Proof of Proposition \ref{vanishingLyap}]
The implication $(2)\Rightarrow (1)$ is trivial. Suppose on the other hand that for almost every $\omega\in \Omega$ the family $\mathcal M_\omega$ is bounded.

Given $\omega\in \Omega$, we define $f(\omega)=\log|\det(M_{\omega})|$. By elementary properties of the determinant, we obtain that
\begin{align*}
\log|\det(M^n_\omega)|=\sum_{i=0}^{n-1}f(T^i\omega).
\end{align*}
Let $Y_i:=f(T^{i-1}\omega)$ and $X_n=\sum_{i=1}^n Y_i$. The $Y_i$-s form a sequence of i.i.d. random variables with expected value $\E(Y_i)=0$. If $\sqrt{Var(Y_i)}>0$, almost surely there exists a sequence $n_k$ so that $X^{n_k}\to\infty$, contradicting the fact that $\mathcal M_\omega$ is almost certainly bounded. It follows that $Var(Y_i)=0$, which implies that $|\det(M^n_\omega)|=1$ for all $\omega\in \Omega$ and $n\in\mathbb N$.

Let $\Omega_a$ as in the Lemma \ref{fullmeasure} and choose $\omega_0\in \Omega_a$ for which $\mathcal M_{\omega_0}$ is bounded. We write $M^{-n}_{\omega_0}=(M^n_{\omega_0})^{-1}$. Since $|\det(M^n_\omega)|=1$ for every $n$, the extended family $\widehat{\mathcal M}_{\omega_0}:=\{(M^{n}_{\omega_0})\}_{n\in \mathbb Z}$ is also bounded. Let $C$ be a bound on the norms of the matrices of this family.

Let $\omega\in\Omega$ and $n\in\mathbb N$. Since $\omega_0\in\Omega_a$, there exists $k_j$ such that $M^n_{T^{k_j}\omega_0}\to M^n_\omega$. Now
\[
\Vert M^n_{T^{k_j}\omega_0}\Vert=\Vert M^{n+k_j}_{\omega_0}\cdot M^{-k_j}_{\omega_0}\Vert<C^2.
\]
This proves that the set $S_\nu$ is bounded in $GL(m,\mathbb C)$. Its closure $G_\nu=\overline{S_\nu}$ is compact in $GL(m,\mathbb C)$.

Given $M\in G_\nu$, its orbit is also contained in $G_\nu$, thus it is bounded. Since $|\det(M)|=1$ the matrix $M$ is diagonalizable and every eigenvalue of $M$ has norm $1$. In this case there exists $n_k$ such that $M^{n_k}\to \textrm{id}$. Furthermore $M^{n_k-1}$ is a convergent sequence and its limit coincide with $M^{-1}$, hence $G_\nu$ is a compact subgroup of $GL(m,\mathbb C)$. By the previous lemma it follows that $G_\nu$ is conjugated to a subgroup of $U(m)$.
\end{proof}
\subsection{Semi-neutral measures.}\mbox{}

\medskip\noindent
Throughout the rest of this section we assume that $\nu$ is semi-neutral and with compact support, and consider the induced i.i.d. cocycle. Almost surely $M^n_\omega\not = 0$ for every $n$. Let $\mathbb P\Omega=\Omega\times \mathbb P^{m-1}$ and $\mathbb PF:\mathbb P\Omega\rightarrow\mathbb P\Omega$ be the map defined by
\[
\mathbb PF:(\omega,[v])\rightarrow \left(T\omega, [M_\omega v]\right).
\]

Finally let $\Phi:\mathbb P\Omega\rightarrow \mathbb R$ be the map $\Phi(\omega,[v])=\log\left(\Vert M_\omega v\Vert/\Vert v\Vert\right)$. For every $\omega\in\Omega$ and $v\in\mathbb C^m$ we have
\begin{equation}
\label{inducedform}
\sum_{k=0}^{n-1}\Phi\circ \mathbb P F^k(\omega,[v])=\log\frac{\Vert M^n_\omega v\Vert}{\Vert v\Vert}\le \log \Vert M^n_\omega\Vert.
\end{equation}

Let $0=\kappa_1>\dots>\kappa_s$ be the Lyapunov indices of the measure $\nu$ and $\mathbb P^{m-1}=\mathcal V_1(\omega)\supset\dots \supset\mathcal V_s(\omega)$ be the collection of vector subspaces introduced in Theorem \ref{odeselecthm}, defined for almost every $\omega\in\Omega$. We write $\mathbb P\Omega_1,\dots,\mathbb P\Omega_s$ for the family of disjoint subsets of $\mathbb P\Omega$ given by
\[
\mathbb P\Omega_j=\{(\omega,[v])\,:\,v\in \mathcal V_j(\omega)\setminus \mathcal V_{j-1}(\omega)\}.
\]
We recall the following statement:
\begin{theorem}[\cite{Led},\cite{Via}]
\label{inducedmeasure}
Given any $\mathbb PF$-invariant ergodic probability measure $m$ on $\mathbb P\Omega$ that projects down to $\mu=\nu^{\infty}$, there exists $j\in\{1,\dots,s\}$ such that
\begin{equation}
\label{ego}
\int\Phi\,dm=\kappa_j,\quad\text{and}\quad m\left(\mathbb P\Omega_j\right)=1.
\end{equation}
Conversely, given $j\in\{1,\dots,s\}$ there exists a $\mathbb PF$-invariant ergodic probability  measure $m_j$ projecting to $\mu$ and satisfying \eqref{ego}.
\end{theorem}

The following lemma closely resembles Lemma 3.6 of \cite{Gu}, we provide a proof for the sake of completeness.
\begin{lemma}
Let $(X,T,\mu)$ be an ergodic dynamical system and $f:X\rightarrow \mathbb R$ an integrable function. Suppose that $\int_Xf\,d\mu=0$, then almost surely there exists a sequence $n_k$ such that
\[
\lim_{k\to\infty}\sum_{i=0}^{n_k-1}f\circ T^ix=0
\]
\end{lemma}
\begin{proof}
By replacing $(X,T,\mu)$ with its natural extension if necessary, we may assume that $T$ is invertible.
We consider the product $Y=X\times \mathbb R$ and the map $S:Y\rightarrow Y$ defined as
\[
S(x,r):=(Tx,r+f(x)).
\]
The transformation $S$ preserves the measure $\nu=\mu\times \lambda$, where $\lambda$ is the Lebesgue measure of $\mathbb R$. We can write
\[
S^n(x,r)=\left(T^nx,r+\sum_{i=0}^{n-1}f\circ T^ix\right)=(T^nx,s_n(x,r)).
\]

Suppose that there exists a \textit{wandering} subset $A\subset Y$, i.e. $\nu(A)>0$ and for every couple $i\neq j$ the sets $S^i(A)$ and $S^j(A)$ are disjoint, up to a set of measure $0$. By ergodicity of $\mu$, for almost every $(x,r)\in Y$ we have
\[
\lim_{n\to\infty}\frac{s_n(x,r)}{n}=0.
\]
We can choose $B\subset A$ of finite positive measure and such that the limit above holds on $B$ uniformly. Since $\nu(\bigcup_{0\le i\le n} S^i(B))\le \lambda (\bigcup_{0\le i\le n} s_i(B))$, it follows that $\lim_n \nu(\bigcup_{0\le i\le n} S^i(B))/n=0$. On the other hand, since $A$ is wandering and $\nu$ is $S$-invariant we find that $\nu(\bigcup_{0\le i\le n} S^i(B))/n=\nu(B)>0$; which contradicts the assumption on $A$.

Let $\varepsilon>0$ and $Y_{\varepsilon}=X\times[-\varepsilon,+\varepsilon] $. We write $C_\varepsilon$ for the set of all points of $Y_\varepsilon$ which do not return to $Y_\varepsilon$, under iterations of the map $S$. The map $S$ is invertible, therefore if $\nu(C_\varepsilon)>0$ then the set $C_\varepsilon$ is wandering. Since there are no wandering sets, it follows that $\nu(C_\varepsilon)=0$.

Let $X_\varepsilon\subset X$ be the set of all points $x$ for which there exists $n>0$ such that $-\varepsilon\le s_n(x,0)\le \varepsilon$. We note that $(X\setminus X_\varepsilon)\times [-\varepsilon/2,\varepsilon/2]\subset C_{\varepsilon/2}$, thus $X_\varepsilon$ is a full measure subset of $X$. The set
\[
X_0=\bigcap_n T^{-n}\left(\bigcap_k X_{1/k}\right)
\]
is the intersection of countable many full measure sets, thus it has full measure. Take $x_0\in X_0$,  since $X_0\subset \cap _kX_{1/k}$, then for every $k$ there exists a minimal $n_k>0$ such that $-1/k\le s_{n_k}(x_0,0)\le 1/k$. We notice that $n_{k+1}\ge n_k$, thus either $n_m\to\infty$, in which case the lemma is proved, or there exists $n'_1$ such that $s_{n'_1}(x_0,0)=0$. If the second case occurs, we let $x_1=T^{n'_1}x_0\in X_0$ and we repeat the argument for $x_1$. It follows that one can always construct a sequence which satisfies the requirements of the Lemma.
\end{proof}

\begin{proposition}
\label{derivativedoesnotvanish}
Let $\nu$ be a semi-neutral measure with compact support. Suppose that $\mathcal M_\omega$ is almost surely bounded, then for every $M^n_\omega\in S_\nu$ we have
\[
\Vert M^n_\omega\Vert\ge 1.
\]
\end{proposition}
\begin{proof}
Given $K>0$, let
\[
\Omega_K:=\left\{\omega\in\Omega\,\big|\, \sup_n\log\Vert M^n_\omega\Vert\le K\right\}.
\]
The set $\Omega_b=\bigcup_{K>0}\Omega_K$ coincides with the set of all $\omega$ for which $\mathcal M_\omega $ is bounded. It then follows that $\lim_{K\to\infty}\nu(\Omega)=1$ and that we can choose $K_0>0$ so that $\nu(\Omega_{K_0})>0$.

Suppose that there exists $M^{N}_{\omega_0}\in S_\nu$ with $\Vert M^N_{\omega_0}\Vert<1$.
By the semigroup structure of $S_\nu$ we may further assume that $\Vert M^N_{\omega_0}\Vert \le 1/(3K_0)$.

For $i=1,\dots,N$ we choose $U_i\subset Mat(m,\mathbb C)$ such that the set $\mathcal C_{U_1,\dots U_N}$ has positive $\mu$-measure and $\Vert M^N_\omega\Vert \le 1/(2K_0)$ for every $\omega\in\mathcal C_{U_1,\dots U_N}$.

We define the set
\[
\mathcal U:=\mathcal C_{U_1,\dots U_N}\cap T^{-N}(\Omega_{K_0})=U_1\times \dots \times U_N\times \Omega_{K_0}.
\]
It is clear that this set has positive $\mu$-measure.

Let $m_1$ be the $\mathbb PF$-invariant ergodic measure on $\mathbb P\Omega$ defined in Theorem \ref{inducedmeasure}. Given $(\omega,[v])\in\mathcal U\times \mathbb P^{m-1}$, for every $n\ge N$ we have
\begin{align*}
\sum_{k=0}^{n-1}\Phi\circ \mathbb P F^k(\omega,[v])&\le
\log\Vert M^n_\omega\Vert\\
&\le\log \Vert M^N_\omega\Vert\Vert M^{n-N}_{\sigma^N(\omega)}\Vert\\
&\le -\log 2.
\end{align*}
By \eqref{ego} we have $\int_\Omega \Phi\,dm_1\not = 0$ and $m_1(\mathcal U\times \mathbb P^{m-1}) =\mu(\mathcal U)>0$, contradicting the previous Lemma. It follows that for every $M^n_\omega\in S_\nu$ we have $\Vert M^n_\omega\Vert\ge 1$.
\end{proof}

\section{Compositions of random germs}
\label{section:generalprob}
Let $\mathcal O(\mathbb C^m,0)$ be the space of all germs of holomorphic functions fixing the origin. This set can be endowed with the so-called inductive limit topology $\tau_{ind}$, see the appendix of this paper for more details. Let $\nu$ be a Borel probability measure on the set $\mathcal O(\mathbb C^m,0)$. Let $\Omega$ be the space of all sequences in $supp(\nu)$ and $\mu=\nu^\infty$. Let $T:\Omega\rightarrow\Omega$ denote the left shift function and $f:\Omega\rightarrow \mathcal O(\mathbb C^m,0)$ be the function which returns the first element of the sequence. As we will see later, the function $f$ defines a (non linear) cocycle over $T$. In analogy with the previous section we will write
\[
f^n_\omega=f_{T^{n-1}\omega}\circ \cdots\circ f_\omega,
\]
and $\mathcal F_\omega=\{f^n_\omega\}$. Finally we define
\[
\Omega_{nor}=\{\omega\in\Omega\,|\, \mathcal F_\omega\textrm{ is a normal family near the origin}\}.
\]

\begin{proposition}
The set $\Omega_{nor}$ is measurable.
\end{proposition}
\begin{proof}
Normality near the origin is equivalent to relatively compactness of the set $\mathcal F_\omega$ with respect to $\tau_{ind}$. Furthermore by Corollary \ref{compactset3} the set $\mathcal F_\omega$ is relatively compact in $\mathcal O(\mathbb C^m,0)$ if and only if $\mathcal F_\omega\subset\overline B_n(0,r)$ for some natural number $n$ and $r>0$. The terminology $\overline B_n(0,r)$ is introduced in the appendix.

Let $r_n$ be an increasing sequence of positive real numbers such that $r_n\to \infty$. Since $\overline B_n(0,r)\subset \overline B_{n+1}(0,r)$, it follows that
\begin{align*}
\Omega_{nor}&=\bigcup_n\{\omega\in\Omega\,|\, \mathcal F_\omega\subset\overline B_n(0,r_n)\}\\
&=\bigcup_n\bigcap_m \{\omega\in\Omega\,|\,f^m_\omega\in \overline B_n(0,r_n)\}.
\end{align*}
The function $\omega\mapsto f^m_\omega$ is measurable for every $m$. By Lemma \ref{relatcomp} and Corollary \ref{compactset3},  $\overline B_n(0,r_n)$ is a closed set in the inductive limit topology. This proves that $\Omega_{nor}$ is the intersection of countably many measurable sets, and thus that it is a measurable set.
\end{proof}
\medskip

\noindent{\bf Problem}\emph{ Let $\omega\in\Omega$ be a random sequence of germs. What is the probability that $\mathcal F_{\omega}$ is a normal family?}
\medskip

We notice that the set $\Omega_{nor}$ is backward invariant. By ergodicity of $T$ it follows that
\begin{proposition}
Normality of $\mathcal F_{\omega}$ occurs either with probability $1$ or $0$, depending on the probability measure $\mu$.
\end{proposition}

\begin{definition}[Random Fatou set]
We say that the origin belongs to the \textit{(local) random Fatou set} if $\Omega_{nor}$ has full measure. If this is not the case we say that the origin belongs to the \textit{(local) random Julia set}.
\end{definition}
Suppose now that the measure $\nu$ has compact support. By Corollary \ref{compactset3}, there exists $R>0$ such that each $f\in supp(\nu)$ is holomorphic and bounded on $\mathbb B_R$, the open ball of $\mathbb C^m$ with radius $R$ and center the origin. Given $\omega\in\Omega$ we may define the set
\begin{equation}
\label{invariancerel}
D_{\omega}:=\bigcap_{n=0}^\infty \left(f^n_\omega\right)^{-1}(\mathbb B_R)
\end{equation}
Clearly $0\in D_{\omega}$, thus $D_{\omega}$ is not empty. The following invariance relation holds $$ f^n_\omega\left(D_{\omega}\right)=D_{T^n(\omega)}.$$

Given $z\in D_\omega$ its orbit $\{f^n_\omega(z)\,|\,n\in\mathbb N\}$ is well defined and bounded. By the weak Montel Theorem the family $\mathcal F_\omega$ is normal on $\textrm{int}( D_\omega)$. We will call this set the \textit{(local) Fatou set} of the sequence $\omega$ and we will denote it as $F_\omega$. It is clear that the origin lies in the random Fatou set if and only if it  almost certainly lies in the Fatou set of $\omega$.

We define
\[
D:=\{(\omega,z)\in\Omega\times \mathbb C^m\,|\, z\in D_\omega\}.
\]
The \textit{(non linear) cocycle} defined by $f$ over $T$ is the skew-product transformation $F: D\rightarrow D$
\[
F(\omega,z)=(T\omega,f_\omega(z)),\qquad F^n(\omega,z)=(T^n\omega,f^n_\omega(z)).
\]

By Lemma \ref{derivation} the map $d_0:f\mapsto df(0)$ is a continuous map. Following the previous section, we introduce the following classification of probability measures on $\mathcal O(\mathbb C^m,0)$:
\begin{definition}
Let $\nu$ be a probability  measure on $\mathcal O(\mathbb C^m,0)$ so that $(d_0)_*\nu$ has compact support. We say that the measure $\nu$ is \textit{attracting, repelling, neutral or semi-neutral} if $(d_0)_*\nu$ is respectively attracting, repelling, neutral or semi-neutral.
\end{definition}

We are now ready to state the central theorem of this paper. Notice that the neutral case coincides with Theorem \ref{neutralmeasure} of the introduction .
\begin{theorem}
\label{theorem:maintheorem}
Let $\nu$ be a probability measure on $\mathcal O(\mathbb C^m,0)$ with compact support, and suppose that $\nu$ is either attracting, repelling or neutral.
\begin{itemize}
\item[]\textbf{Attracting:} the origin lies in the random Fatou set, and it is almost surely an attracting point for the system $\{f^n_\omega|_{D_{\omega}}\}$.
\item[]\textbf{Repelling:} the origin lies in the random Julia set.
\item[]\textbf{Neutral:} the origin lies in the random Fatou set if and only if all the germs in $supp(\nu)$ are simultaneously linearizable, and the semigroup of the differentials
\[
dS_\nu:=\{df^n_\omega(0)\,|\,\omega\in \Omega,\, n\in \mathbb N\},
\]
is conjugate to a sub-semigroup of $U(m)$.
\end{itemize}
\end{theorem}

If we drop the compactness hypothesis on $supp(\nu)$, then this theorem is false in general. We will show a counterexample, for an attracting measure $\nu$, with non compact support.

\begin{example}
Let $0<\lambda_1<1$ and let $a_i=\lambda_1^{-3\cdot 2^i}$. We consider the family
\[
\mathcal A:=\{f_i(z)=\lambda_1z+a_iz^2\}.
\]
Let $\nu$ be the probability measure on $\mathcal O(\mathbb C,0)$ such that $\nu(f_i)=\frac{1}{2^i}$. The support of $\nu$ coincides with $\mathcal A$ and is not compact. On the other hand we have $supp((d_0)_*\nu)=\{\lambda\}$, which implies that the measure $\nu$ is attracting. In this case the (unique) Lyapunov index is $\kappa_1=\log \lambda_1<0$.

Let $\omega=(f_{i_1},f_{i_2},\dots)$ be a sequence of i.i.d. random germs of $\mathcal A$. We write $\lambda_2^{(n)}$ for the coefficient relative to $z^2$ of $f^n_\omega$. It follows that
\[
\lambda_2^{(n)}=\lambda_1^{n}\sum_{k=1}^n\lambda_1^{k-1}a_{i_k}\geq \lambda_1^{2n}a_{i_n}.
\]
Write $p_n$ for the probability that $\lambda_1^{2n}a_{i_n}\geq \lambda_1^{-n}$. The inequality holds if and only if $i_n\geq \log n$, thus
\[
p_n=\sum_{k\geq \log n}\frac{1}{2^k}\sim \frac{2}{n}.
\]
Given $\omega$ and $m>0$, by the independence assumption, the probability that for every $n\geq m$ the inequality $\lambda_1^{2n}a_{i_n}< \lambda_1^{-n}$ holds is equal to
\[
P_m=\prod_{n\geq m}(1-p_n)=\textrm{Exp}\left(\sum_{n\ge m}\log(1-p_n)\right).
\]
Since the general term of the series above is asymptotic to $-2/n$ we see that $\sum_{n\ge m}\log(1-p_n)=-\infty$ and therefore $P_m=0$.
It follows that almost surely there exists $n_k\to\infty$ such that $\lambda_2^{(n_k)}\to\infty$, thus that almost surely $\mathcal F_{\omega}$ is not a normal family.
\end{example}

We conclude this section with the discussion of a particular semi-neutral measure. The example shows that, in the semi-neutral case, $supp(\nu)$ may contains germs which are not linearizable but at the same time the origin can be in the random Fatou set. Furthermore the semigroup $S_\nu$ generated by $supp(\nu)$ may not be bounded in $\mathcal O(\mathbb C^m,0)$. Thus the methods developed in section \ref{section:neutral} for the neutral case, do not work for the semi-neutral one.
\begin{example}
Suppose that $f$ and $g$ are given by
\begin{align*}
f(z,w)&=(z,w/2),\\
g(z,w)&=(z+zw,w).
\end{align*}
The $n$-th iterate of $g$ has the form $g^n(z,w)=(z+nzw+O(|z|^3),w)$. Therefore $\{g^n\}$ is not normal on any neighborhood of the origin, thus it cannot be linearized.

Let $\nu:=\frac{1}{2}\delta_{f}+\frac{1}{2}\delta_{g}$ be a probability measure. A short computation gives $\kappa_1=0$ and $\kappa_2=\frac{1}{2}\log\frac{1}{2}<0$.

Let $\omega\in\Omega$ be a sequence of i.i.d. random germs and write $f^n_\omega=(f_{\omega,1}^n\,,\,f_{\omega,2}^n)$, then
\begin{align*}
|f_{\omega,1}^{n+1}(z,w)|&\le\begin{cases}|f_{\omega,1}^n(z,w)|&\textrm{if }f_n=f;\\
|f_{\omega,1}^n(z,w)|\left(1+|f^n_{\omega,2}(z,w)|\right)&\textrm{if }f_n=g;
\end{cases}\\
|f_{\omega,2}^{n+1}(z,w)|&=\begin{cases}\frac{1}{2}|f_{\omega,2}^n(z,w)|&\textrm{if }f_n=f;\\
|f_{\omega,2}^n(z,w)|&\textrm{if }f_n=g.
\end{cases}
\end{align*}
Let $\alpha(n)=\#\left\{k\in \{1,\dots, n\} \, |\, f_k=f\right\}$. Using the estimates above we find that
\begin{align*}
|f_{\omega,2}^n(z,w)|&=\frac{|w|}{2^{\alpha(n)}},\\
|f_{\omega,1}^n(z,w)|&\le |z|\prod_{i=1}^{n}\left(1+\frac{|w|}{2^{\alpha(i)}}\right).
\end{align*}
It is a well known fact that almost surely we have $\alpha(n)\sim n/2$. It follows that
\[
\log|f_{\omega,1}^n(z,w)|\le\log|z|+ \sum_{i=1}^n \log\left(1+\frac{|w|}{2^{\alpha(i)}}\right).
\]
As $n\to \infty$, almost certainly $ a_i:=\log\left(1+\frac{|w|}{2^{\alpha(i)}}\right)\sim\frac{|w|}{2^{i/2}}$, therefore $\sum_{i=1}^\infty a_i<\infty$  and $f^n_\omega(z,w)$ is uniformly bounded for every $(z,w)\in \mathbb C^2$. We conclude that in this case $\mathcal F_\omega$ is normal for almost every sequence $\omega$.
\end{example}

\section{Proof of Theorem \ref{theorem:maintheorem}}
We start this section with the simpler attracting and repelling cases, and will cover the neutral case afterwards.
\subsection{Attracting and repelling measures.}
\label{section:nonzero}\mbox{}
\medskip

\noindent
If the measure $\nu$ is repelling, then the sequence of the differentials $df^n_\omega(0)$ diverges almost surely. If this is the case, the origin lies in the random Julia set.

The following statement implies the attracting case. It is likely well known, we provide a proof for the sake of completeness.
\begin{proposition}
\label{attractingmeasure}
Let $\nu$ be an attracting measure on $\mathcal O(\mathbb C^m,0)$ with compact support, then the origin lies in the random Fatou set. Furthermore, there exists almost surely a neighborhood of the origin on which all orbits converge to the origin.
\end{proposition}
\begin{proof}
By Lemma \ref{lyapexpandind} and \eqref{Furstthm}, there exists $n_0>0$ such that $\E \log \Vert df^{n_0}_\omega(0) \Vert<0$. Without loss of generality we may assume that $n_0=1$.

Given $\delta>0$ we write
\begin{align*}
\E\log\Vert df(0)\Vert&=\int_{\Vert df(0)\Vert\le\delta}\log\Vert df(0)\Vert d\nu+\int_{\Vert df(0)\Vert>\delta}\log\Vert df(0)\Vert d\nu\\
&=L_\delta+U_\delta.
\end{align*}
Suppose that $U_\delta<0$ for some $\delta>0$. In this case we can find $\varepsilon>0$ sufficiently small so that
\begin{equation}
\label{epsilonin}
\E\log\left(\Vert df(0)\Vert+\varepsilon\right)<0.
\end{equation}
On the other hand, suppose that $U_\delta\ge 0$ for every $\delta>0$. Let $\delta_n$ be a decreasing sequence of positive numbers converging to $0$ and such that $\delta_0>\sup_{f\in supp(\nu)}\Vert df(0)\Vert$. It then follows that
\begin{align*}
\int_{ df(0)\neq 0}\log \Vert df(0)\Vert d\nu&=\sum_{n=1}^{\infty}\int_{\delta_n<\Vert df(0)\Vert \le\delta_{n-1}}\log\Vert df(0)\Vert d\nu\\
&=\lim_{N\to\infty}U_{\delta_N}\\
&\ge 0.
\end{align*}
Since $\E\log\Vert df(0)\Vert<0$ we must have $M_0:=\nu\{ df(0)=0\}>0$. We conclude that
\begin{align*}
\E\log\left(\Vert df(0)\Vert+\varepsilon\right)&= M_0\log\varepsilon+\int_{\Vert df(0)\Vert\neq 0}(\log\left\Vert df(0)\Vert+\varepsilon\right) d\mu(f)\\
&\le M_0\log\varepsilon+\sup_{f\in supp(\mu)}\log \left(\Vert df(0)\Vert+\varepsilon\right).
\end{align*}
Therefore also in this second case we can choose $\varepsilon>0$ sufficiently small so that \eqref{epsilonin} is satisfied.

Let $\alpha_f:=\Vert df(0)\Vert+\varepsilon$. By the compactness of $supp(\nu)$ there exists $0<r<R$ such that every $f\in supp(\nu)$ is holomorphic on $\mathbb B_R$ and
\[
\Vert f(z)\Vert\le \alpha_f\Vert z\Vert,\qquad \forall z\in \mathbb B_r.
\]
Since $\E\log\alpha_f<0$, then for almost every sequence $\omega\in \Omega$ we have $\alpha_\omega^n:=\alpha_{f_{T^{n-1}\omega}}\cdots \alpha_{f_\omega}\to 0$. In particular there exists $0<\delta\le r$, which depends on $\omega$, such that if $\Vert z\Vert <\delta$ then $\alpha_\omega^n|z|<r$ for every $n>0$.

Let $\Vert z\Vert<\delta$, a simple induction argument shows that
\[
\Vert f^n_\omega(z)\Vert\le\alpha^n_\omega\Vert z\Vert<r,\qquad \forall n\in\mathbb N.
\]

This shows that $\mathbb B_\delta\subset D_\omega$ and that the family $\mathcal F_\omega$ is normal at the origin. This is true for almost every $\omega\in\Omega$, therefore the origin lies in the random Fatou set. Furthermore since $\alpha^n_\omega\to 0$ as $n\to\infty$, the orbits of all points in $\mathbb B_\delta$ converge to the origin.
\end{proof}

\subsection{Neutral measures}
\label{section:neutral}\mbox{}
\medskip

\noindent
The following lemma is proved along the lines of Lemma \ref{fullmeasure}.
\begin{lemma}
Let $\nu$ be a probability measure on $\mathcal O(\mathbb C^m,0)$ with compact support. Then the set
\[
\Omega_{a}:=\left\{\omega\in \Omega\,\Big|\,\forall n\in\mathbb N \textrm{ and }\forall\alpha\in \Omega,\,\exists k_j\,:\, f^n_{T^{k_j}\omega}\to f^n_\alpha\right\}
\]
is a full measure subset of $\Omega$.
\end{lemma}
\begin{proof}
Unlike the standard topology of $Mat(m,\mathbb C)$, the inductive limit topology is not metrizable. However, given $n\in\mathbb N$, the image of $\Omega$ under the continuous map $\omega\mapsto f^n_\omega$ is a compact set in the inductive limit topology. By Corollary \ref{compactset3}, we can find  a natural number $N_n$ so that, for every $\omega\in\Omega$, the germ $f^n_\omega$ belongs to $X_{N_n}$.

We choose a sequence $\varepsilon_j\to 0$ and we define
\[
\Omega_{n,j}=\{\omega\in\Omega\,|\,\forall \alpha\in \Omega,\,\exists k\,:\, d_{N_n}(f^n_\omega,f^n_\alpha)<\varepsilon_j\}.
\]
By Theorem \ref{convergentsequences}, we have $\Omega_a=\cup_{n,j}\Omega_{n,j}$. From this point on we follows the proof of Lemma \ref{fullmeasure}, replacing the standard metric of $Mat(m,\mathbb C)$ with the metric $d_{N_n}$.
\end{proof}

\begin{theorem}[Hurwitz Theorem]
Suppose that $(f_n)_n$ is a sequence of injective holomorphic maps on a domain $U\subset \mathbb C^m$ converging uniformly on compact subsets to a function $g$. Then $g$ is either injective or degenerate.
\end{theorem}

Given a probability measure $\nu$ with compact support, choose $R>0$ and define $D_\omega$ as explained in section \ref{section:generalprob}.
\begin{lemma}
\label{trappinglemma}
Let $\nu$ be a neutral measure with compact support. The origin lies in the random Fatou set if and only if given $\varepsilon>0$ there exists $\rho>0$ such that, for every $\omega\in \Omega$, we have $\mathbb B_{\rho}\subset D_\omega$ and
\begin{equation}
\label{uniformestimate}
f^n_\omega(\mathbb B_\rho)\subset \mathbb B_\varepsilon,\qquad \forall n\in\mathbb N.
\end{equation}
\end{lemma}
\begin{proof}
Suppose that for every $\varepsilon>0$ there exists such $\rho$, then by the weak Montel Theorem the origin lies in the random Fatou set.

Suppose on the other hand that the origin lies in the random Fatou set and let $\omega_0\in\Omega_a\cap\Omega_{nor}$.
By Ascoli--Arzel\'a theorem, given  $\varepsilon>0$ there exists $\delta_0>0$  such that $\mathbb B_{\delta_0}\subset D_{\omega_0}$ and
\[
 f^n_{\omega_0}(\mathbb B_{\delta_0})\subset \mathbb B_\varepsilon,\qquad\forall n\in\mathbb N.
\]
By the compactness of $supp(\nu)$, by choosing $\varepsilon>0$ small enough, we may assume that every $f\in supp(\nu)$ is holomorphic and injective on $\mathbb B_{\varepsilon}$. Therefore $f^n_{\omega_0}:\mathbb B_{\delta_0}\rightarrow \mathbb B_\varepsilon$ is injective for all values of $n$.

Suppose that there exists $n_k$ so that
\begin{equation}
\label{contradictionform}
\mathbb B_{1/k}\not\subset f^{n_k}_{\omega_0}(\mathbb B_{\delta_0}).
\end{equation}
By taking a subsequence if necessary, we may assume that $f^{n_k}_{\omega_0}\to g$ uniformly on compact subsets of $\mathbb B_{\delta_0}$. By Proposition \ref{vanishingLyap} we have $|\det(df^{n_k}_{\omega_0}(0))|=1$ for every $k$, thus $|\det(dg(0))|=1$. Thanks to Hurwitz Theorem we can conclude that $g$ is injective on $\mathbb B_{\delta_0}$.

Since $g$ is an open function, we can choose $\rho^\prime>0$ such that $\mathbb B_{\rho^\prime}\subset g(\mathbb B_{\delta_0})$ and thus such that $\mathbb B_{\rho^\prime}\subset f^{n_k}_{\omega_0}(\mathbb B_{\delta_0})$ for $k$ big enough, which contradicts \eqref{contradictionform}.


It follows that there exists $\rho_0>0$ so that $ \mathbb B_{\rho_0}\subset f^n_{\omega_0}(\mathbb B_{\delta_0})$ for every $n$. By the invariance relation \eqref{invariancerel}, given $n,k\in\mathbb N$, we find that $\mathbb B_{\rho_0}\subset  D_{T^k\omega_0}$ and
\begin{align*}
f^n_{T^k\omega_0}(\mathbb B_{\rho_0})&\subset f^n_{T^k\omega_0}\circ f^k_{\omega_0}(\mathbb B_{\delta_0})\\
&\subset\mathbb B_\varepsilon.
\end{align*}

Let $\rho<\rho_0$ be fixed. Given $\omega\in\Omega$ and $n\in\mathbb N$, there exists a sequence $k_j$ so that $f^n_{T^{k_j}\omega_0}$ converges to $f^n_\omega$ in the inductive limit topology. On the other hand, by the weak Montel Theorem, we may also assume that $f^n_{T^{k_j}\omega_0}$ converges to some $g$, uniformly on compact subsets of $\mathbb B_{\rho_0}$. By Theorem \ref{convergentsequences} the maps $g$ and $f^n_\omega$ agree on a small ball containing the origin, therefore, by the identity principle, they coincide as germs. We conclude that $f^n_\omega$ is holomorphic on $\mathbb B_\rho$ and, since $\rho$ is independent from $n$ and $\omega$, that $\mathbb B_\rho\subset D_\omega$ for every $\omega\in\Omega$. Finally uniform convergence on compact subsets of $\mathbb B_{\rho_0}$ implies \eqref{uniformestimate}.
\end{proof}

We write $S_\nu$ for the semigroup
\[
S_\nu:=\left\{f^n_\omega\,|\,\omega\in \Omega,\,n\in\mathbb N\right\}.
\]
The following corollary is an immediate consequence of the previous lemma.
\begin{corollary}
\label{trappingcor}
Let $\nu$ be a neutral measure with compact support. Then the origin lies in the random Fatou set if and only if  $S_\nu$ is relatively compact in $\mathcal O(\mathbb C^m,0)$.
\end{corollary}

\begin{definition}
Given $f\in \mathcal O(\mathbb C^m,0)$ we say that $f$ is \emph{linearizable}, if there exists $\varphi\in \mathcal O(\mathbb C^m,0)$, locally invertible at the origin, such that
\[
\varphi\circ f(z)=df(0)\cdot \varphi.
\]
If this is the case, we say that $f$ is \emph{linearized} by $\varphi$.
\end{definition}
\begin{corollary}
\label{finitecomplinear}
Let $\nu$ be a neutral measure with compact support. If the origin lies in the random Fatou set, then every element of $S_\nu$ is linearizable.
\end{corollary}
\begin{proof}
If the semigroup $S_\nu$ is relatively compact in $\mathcal O(\mathbb C^m,0)$ then, given $f\in S_\nu$, the family $\{f^n\}$ is also relatively compact. It follows that the origin lies in the Fatou set of $f$, and therefore that the germ is linearizable.
\end{proof}

We notice that linearizability of every element of $S_\nu$ does not imply that the origin lies in the random Fatou set. Consider the following example.
\begin{example}
Let $\lambda=e^{2\pi i\alpha}$ be an irrational rotation, with $\alpha$ Brjuno number. We define the maps
$$
f_1(z) = \lambda(z + z^2), \; \; \mathrm{and} \; \;  f_2(z) = \lambda(z - z^2),
$$
and consider a measure such that $supp(\nu)=\{f_1,f_2\}$.

In this case we see immediately that every function in the semigroup $S_\nu=\langle f_1,f_2\rangle$ is linearizable. Nevertheless $f_1$ and $f_2$ are not simultaneously linearizable.

Let $z_0 \neq 0$ and define integers $n_i \in \{1,2\}$ recursively as follows: if $\mathrm{Re} (z_i) \ge 0$ then $n_{i+1} = 1$, else $n_{i+1} = 2$. Let $\omega:=(f_{n_1},f_{n_2},\dots)$, and $z_i=f^i_\omega(z_0)$.

We claim that the orbit of $z_0$ converges to infinity. To see this suppose that $\mathrm{Re} (z_i) \ge 0$. Then it follows that the angle between the vectors $z_i$ and $z_i^2$ is at most $\pi/2$, and hence
$$
|z_{i+1}| = |z_i + z_i^2| \ge |z_i|.
$$
The irrational rotation guarantees that the norm increases often enough to converge to infinity, proving the claim.

The conclusion is that in this case $S_\nu$ is not relatively compact, therefore the origin does not lie in the random Fatou set.
\end{example}

Suppose that $\nu$ is a neutral measure with compact support for which the origin lies in the random Fatou set. We write $G_\nu$ for the closure of $S_\nu$ in $\mathcal O(\mathbb C^m,0)$.
\begin{lemma}
\label{bochyp}
The set $G_\nu$ is a compact topological group. Moreover there exists an open subset $M\subset \mathbb C^m$ such that every $f\in G_\nu$ belongs to $Aut(M)$.
\end{lemma}
\begin{proof}
By Lemma \ref{trappinglemma} the set $G_\nu$ is compact. Furthermore, by Corollary \ref{finitecomplinear}, every $f\in S_\nu$ is linearizable and, by Proposition \ref{vanishingLyap} the differential $df(0)$ is conjugated to an element of $U(m)$. It follows that for some sequence $n_j$ we have $
f^{n_j}\to \textrm{id}$. This proves that $G_\nu$ contains the identity element. Furthermore  the germ $g=\lim_{j\to\infty}f^{n_j-1}$ is the inverse of $f$ in $\mathcal O(\mathbb C^m,0)$, thus $G_\nu$ contains also the inverse of $f$.

Let $f\in G_\nu$ and $f_n\in S_\nu$ such that $f_n\to f$. Every $f_n$ has an inverse $f_n^{-1}\in G_\nu$. By taking a subsequence if necessary, we may assume that $f_n^{-1}\to g\in G_\nu$. It is not difficult to prove that $g$ is the inverse element of $f$ in $\mathcal O(\mathbb C^m,0)$, which finally proves that $G_\nu$ is a compact group.

Let $\varepsilon>0$ and $\rho>0$ as in Lemma \ref{trappinglemma}. If we choose $\varepsilon$ small enough we may further assume that all the elements of $supp(\nu)$ are univalent on $\mathbb B_\varepsilon$. It follows that every $f\in S_\nu$ is univalent on $\mathbb B_\rho$. Since every $g\in G_\nu$ is invertible, by Hurwitz's theorem, it is also univalent on $\mathbb B_\rho$.

By taking a smaller $\varepsilon$ if necessary, we may further assume that every $g\in G_\nu$ is univalent on $\mathbb B_\varepsilon$. We define the open set
\[
M=\bigcup_{g\in G_\nu}g(\mathbb B_\rho)\subset \mathbb B_\varepsilon.
\]
Given $g\in G_\nu$ and $x\in M$, then $x=\hat g(z)$ for some $\hat g\in G_\nu$ and $z\in \mathbb B_\rho$. It follows that $g(M)\subset M$. Furthermore, since $g^{-1}\in G_\nu$, we obtain the equality $g(M)=M$. Which proves that $g\in Aut(M)$.
\end{proof}

The following result is known as Bochner's Linearization Theorem. A proof of the theorem, valid for $C^k$ diffeomorphism, can be found in \cite{DK}. The same proof is valid also in the holomorphic case, up to small modifications.

\begin{theorem}[Bochner's Linearization Theorem]
Let $M$ be a complex manifold and $x_0\in M$. Let $A$ be a continuous homomorphism from a compact group $K$ to $Aut(M)$ such that $A_k(x_0)=x_0$, for all $k\in K$. Then there exists a $K$ invariant open neighborhood $U$ of $x_0$ in $M$ and biholomorphism $\varphi$ from $U$ onto an open neighborhood $V\subset T_{x_0}M$ of $0$, such that:
\[
\varphi(x_0)=0, \qquad d\varphi(x_0)=id:T_{x_0}M\rightarrow T_{x_0}M
\]
and
\[
\varphi\circ A_k(x)=dA_k(x_0)\cdot \varphi(x)\qquad \forall k\in K,x\in U).
\]
\end{theorem}

\begin{definition}
We say that two or more germs are \emph{simultaneously linearizable} if there exists a locally invertible $\varphi\in \mathcal O(\mathbb C^m,0)$ such that all the germs are linearized by $\varphi$.
\end{definition}

\begin{lemma}
Let $\nu$ be a neutral measure on $\mathcal O(\mathbb C^m,0)$ with compact support. If the origin lies in the random Fatou set then all the elements $f\in supp(\nu)$ are simultaneously linearizable.
\end{lemma}

\begin{proof}
Let $G_\nu$ be the closure of $S_\nu$ and $M$ be the $G_\nu$-invariant open set described in Lemma \ref{bochyp}. The compact group $G_\nu$ induces an action on $M$ that satisfies the hypothesis of Bochner's Linearization Theorem. It follows that all the germs $f\in G_\nu$ can be simultaneously linearized.
\end{proof}

We are finally ready to conclude the proof of Theorem \ref{theorem:maintheorem}.
\begin{proof}[Proof of Theorem \ref{theorem:maintheorem}]
Suppose that the origin lies in the random Fatou set. By the previous lemma it follows that all the germs in $supp(\nu)$ are simultaneously linearizable, which implies that the semigroup $S_\nu$ is conjugated to the semigroup of the differentials $dS_\nu$. Furthermore by Proposition \ref{vanishingLyap}, the semigroup $dS_\nu$ is itself conjugate to a sub-semigroup of $U(m)$. The other implication of the theorem is trivial.
\end{proof}

\section{Discussion of semi-neutral measures}\label{section:semineutral}

Recall from section \ref{section:generalprob} that for every $\omega\in\Omega_{nor}$ the origin belongs to the Fatou set $F_\omega:=\textrm{int}(D_\omega)$. Let $U_\omega$ be the connected component of $F_\omega$ containing the origin.

\begin{proposition}
\label{degenerate}
Let $\nu$ be a semi-neutral measure on $\mathcal O(\mathbb C^m,0)$ with compact support, so that the origin lies in the random Fatou set. Then almost surely, every limit map $g=\lim_{k\to\infty}f^{n_k}_\omega$ is degenerate on $U_\omega$.
\end{proposition}
\begin{proof}
By Corollary \ref{equivalentdefmeas} we have $\E \log |\det df(0)|<0$. Along the lines of Proposition \ref{attractingmeasure}, we can choose $\varepsilon>0$ such that $\E \log (|\det df(0)|+\varepsilon)<0$. In particular for almost every $\omega\in\Omega$ we have
\begin{equation}
\label{zeroproduct}
\prod_{k=1}^\infty \left(|\det df_{T^{k-1}\omega}(0)|+\varepsilon\right)= 0.
\end{equation}

By compactness of $supp(\nu)$, we may choose $r>0$ so that for every $\Vert z\Vert <r$ we have
\[
|\det df(z)|<|\det df(0)|+\varepsilon,\qquad\forall f\in supp(\mu).
\]
Since the origin belongs to the random Fatou set, for almost every $\omega\in\Omega$ there exists $\delta_\omega>0$ such that $\Vert f^n_\omega(z)\Vert<r$ for every $z\in \mathbb B_{\delta_\omega}$ and $n\ge 0$.
We conclude that for $z\in \mathbb B_{\delta_\omega}$
\begin{align*}
|\det df^n_\omega(z)|&=\prod_{k=1}^n|\det df_{T^{k-1}\omega}(f_\omega^{k-1}(p))|\\
&\le \prod_{k=1}^n\left(|\det df_{T^{k-1}\omega}(0)|+\varepsilon\right).
\end{align*}
Let $n_k$ be a sequence such that $f^{n_k}_\omega\to g$ locally uniformly on a neighborhood $U$ of the origin, then on a possibly smaller neighborhood we have $\det dg\equiv 0$. By the identity principle we conclude that the same is true on $U_\omega$.
\end{proof}

Given $z\in U_\omega$ we define the \textit{stable set}
\[
\mathbb W^s_\omega(z):=\{w\in U_\omega\,|\,\Vert f^n_\omega(z)-f^n_\omega(w)\Vert\to 0\}.
\]

When $m=2$, we will prove that, given $\nu$ semi-neutral with compact support and such that the origin lies in the random Fatou set, the stable set through every point $z$ sufficiently close to the origin is locally a complex submanifold. It is a natural question whether the same is true when $m>2$.

Before proceeding to the proof of this result we will present two examples of this phenomenon in the case the germs in $supp(\nu)$ are linear maps.
\begin{example}
Suppose $\nu=\frac{1}{2}\delta_g+\frac{1}{2}\delta_h$, where
\[
g(z,w)=\begin{pmatrix}
1/2 & 0\\ 0 & 1
\end{pmatrix}\begin{pmatrix}
z \\ w
\end{pmatrix},\qquad h(z,w)=\begin{pmatrix}
1/2 & 1\\ 0 & 1
\end{pmatrix}\begin{pmatrix}
z \\ w
\end{pmatrix}.
\]
We notice that given $\omega\in \Omega$
\[
f^{n+1}_\omega(z,w)=\begin{pmatrix}
1/2^n & \alpha_n\\ 0 & 1
\end{pmatrix}\begin{pmatrix}
z \\ w
\end{pmatrix},\qquad\textrm{with }\alpha_n=\begin{cases}\frac{\alpha_n}{2}& \textrm{if }f_{T^n\omega}=g\\
\frac{\alpha_n}{2}+1& \textrm{if }f_{T^n\omega}=h.
\end{cases}
\]
The two Lyapunov indices of the measure $\nu$ are $\kappa_1=0$ and $\kappa_2=-\log 2$, therefore the measure is semi-neutral. Furthermore we have $0\le \alpha_n<2$ for every $n$ and every $\omega$, which implies that the origin lies in the random Fatou set. Given $(z_0,w_0)\in U_\omega=\mathbb C^2$, we can write its stable set as
\[
\mathbb W^s_\omega(z_0,w_0)=\{(z,w)\in\mathbb C^2\,|\, w=w_0\},
\]
which is a one-dimensional manifold. We notice that in this case the stable manifolds are independent from the sequence $\omega\in\Omega$, but that this is not true in general.

Consider for example the measure $\tilde \nu$, obtained by taking instead the maps
\[
g(z,w)=\begin{pmatrix}
1/2 & 0\\ 0 & 1
\end{pmatrix}\begin{pmatrix}
z \\ w
\end{pmatrix},\qquad h(z,w)=\begin{pmatrix}
1/2 & 1\\ 0 & 1
\end{pmatrix}\begin{pmatrix}
z \\ w
\end{pmatrix}.
\]
This measure has the same Lyapunov indices and the origin lies in the random Fatou set. However the sequence $f^n_\omega$ converges, for every $\omega$, to a map of the form $\tilde g(z,w)=(z+\beta_\omega w,0)$. Hence the stable manifolds are again parallel complex planes, but now depend on $\omega$.
\end{example}
\begin{remark}
The last two examples are helpful models in order to understand the random dynamics of semi-neutral measures. We notice that in both cases $df(0)$ and $dg(0)$ share an eigenvalue. In the theory of linear cocycles one says that a cocycle is \textit{strongly irreducible} if there is no finite family of proper subspaces invariant by $M_\omega$ for $\nu$-almost every $\omega$. This raises the following question
\par\medskip\noindent
\textit{Does there exist a semi-neutral measure, such that $M^n_\omega$ is almost surely bounded, but the corresponding cocycle is strongly irreducible?}
\end{remark}
\begin{lemma}
There exists $r>0$ and, for almost every $\omega\in\Omega$, a sequence $n_k$ dependent by $\omega$, so that $f^{n_k}_\omega\to g$ locally uniformly on $U_\omega$ and so that
\[
\mathbb W^s_\omega(z)=\{w\in U_\omega\,|\, g(w)=g(z)\}
\]
for every $z\in U_\omega$ which satisfies $\Vert g(z)\Vert\le r$.
\end{lemma}
\begin{proof}
Given $r>0$ we define the measurable sets
\[
\mathcal A_r:=\{\omega\in\Omega\,|\,\mathbb B_r\Subset U_\omega\}.
\]
It is clear that $\mathcal A_{r_1}\supset \mathcal A_{r_2}$, when $r_1<r_2$, and that $\bigcup_{r>0}\mathcal A_r=\Omega_{nor}$. Since the latter is a full measure set, we may choose $r>0$ so that $\mu(\mathcal A_r)>1/2$.

Let $\varepsilon_1>0$, by the Ascoli--Arzel\`a Theorem the collection of measurable sets
\[
\mathcal B_\delta^{(1)}:=\left\{\omega\in\mathcal A_r\,\Big|\,\begin{array}{c}
\forall z,w\in \mathbb B_r\textrm{ with }\Vert z-w\Vert<\delta :\\
\sup_n\Vert f^n_\omega(z)-f^n_\omega(w)\Vert<\varepsilon_1
\end{array}
\right\},
\]
defined for $\delta>0$, covers $\mathcal A_r$. Furthermore it satisfies $\mathcal B_{\delta}^{(1)}\supset \mathcal B_{\delta'}^{(1)}$, for $\delta<\delta'$, thus we may choose $\delta_1>0$ so that $\mu(\mathcal B_{\delta_1}^{(1)})>1/2$. We will write $\mathcal B^{(1)}$ for the set $\mathcal B^{(1)}_{\delta_1}$.

Given a sequence of positive real numbers $\varepsilon_k$ with $\lim_{k\to\infty}\varepsilon_k=0$, we may recursively define a sequence $\delta_k$ of positive real numbers, in such a way that
\[
\mathcal B^{(k+1)}:=\left\{\omega\in\mathcal \mathcal B^{(k)}\,\Big|\,\begin{array}{c}
\forall z,w\in \mathbb B_r\textrm{ with }\Vert z-w\Vert<\delta_{k+1} :\\
\sup_n\Vert f^n_\omega(z)-f^n_\omega(w)\Vert<\varepsilon_{k+1}
\end{array}
\right\}
\]
has measure strictly greater that $1/2$.

Since the nested sets $\mathcal B^{(k)}$ all have measure at least $1/2$, the intersection, which we will simply denote as $\mathcal B$ satisfies $\mu(\mathcal B)\ge 1/2$.

By ergodicity of the transformation $T$ there exists, for almost every $\omega\in \Omega_{nor}$ a sequence $n_k$ so that $T^{n_k}(\omega)\in\mathcal B$. By taking a subsequence if necessary, we find that $f^{n_k}_\omega$ converges locally uniform on $U_\omega$ to a function $g$.

Given $z\in U_\omega$ such that $\Vert g(z)\Vert<r$, it is clear that $g(w)=g(z)$ for every $w\in \mathbb W^s_\omega(z)$. On the other hand, given $w\in U_\omega$ which satisfies $g(w)=g(z)$, we can find $k_0$ so that $f^{n_k}_\omega(z),f^{n_k}_\omega(w)\in \mathbb B_r$ for every $k\ge k_0$. Furthermore we may also find a sequence $k_0\le k_1\le k_2\le \dots$, so that whenever $k\ge k_i$ we have $\Vert f^{n_k}_\omega(z)-f^{n_k}_\omega(w) \Vert<\delta_i$.

Since $T^{n_k}\omega\in \mathcal B$ for every $k$, we conclude that, given $n\ge n_{k_i}$,
\[
\Vert f^n_\omega(z)-f^n_\omega(w)\Vert=\left\Vert f^{n-n_{k_i}}_{T^{n_{k_i}}\omega}(f^{n_{k_i}}_\omega(z))-f^{n-n_{k_i}}_{T^{n_{k_i}}\omega}(f^{n_{k_i}}_\omega(w))\right\Vert<\varepsilon_i.
\]
This shows that $w\in \mathbb W^s_\omega(z)$, which proves the desired equality.
\end{proof}
When $m=2$ and the measure is semi-neutral we obtain the following (local) version of the Stable Manifold theorem

\begin{corollary}[Stable Manifold Theorem]
Let $\mu$ be a semi-neutral measure on $\mathcal O(\mathbb C^2,0)$ with compact support, for which the origin lies in the random Fatou set. Then there exists, for almost every $\omega\in\Omega$,a constant $\rho>0$ sufficiently small so that, given $z\in \mathbb B_\rho$, the set $\mathbb W^s_\omega(z)$ is locally a one-dimensional complex manifold.
\end{corollary}
\begin{proof}
Let $\omega\in\Omega$ be such that the map $g$ described in the previous lemma exists.
Since $g(0)=0$ we can choose $\rho>0$ sufficiently small so that $\mathbb B_\rho\subset U_\omega$ and $g(\mathbb B_\rho)\subset \mathbb B_r$. Here the value of $r$ is again determined by the previous lemma.

By Proposition \ref{degenerate} the function $g$ is degenerate. Furthermore, by Proposition \ref{derivativedoesnotvanish}, we must have $\Vert dg(0)\Vert\ge 1$, thus we may further assume, by shrinking $\rho$ is necessary, that $g$ has rank $1$ on every point of $\mathbb B_\rho$.

By the constant rank Theorem $g(\mathbb B_\rho)$ is a one-dimensional submanifold and every point in it is a regular value for the holomorphic map $g:\mathbb B_\rho\rightarrow g(\mathbb B_\rho)$. By the Implicit Function Theorem, given $z\in\mathbb B_\rho$, the level set $\{w\in \mathbb B_\rho\,|\, g(w)=g(z)\}$ is a one dimension complex manifold. The level sets of the map $g$ coincide with the stable sets of the sequence $\omega$. It follows that the latter are locally one-dimensional manifolds.
\end{proof}


\appendix
\section{Appendix: Holomorphic germs and their topology}
\label{section:top}

We write $\mathcal O(\mathbb C^m,0)$ for the space of germs at $0$ of holomorphic maps from $\mathbb C^m$ to itself. Note that in the main text we required that the germs fix the origin, but this requirement plays no role here.

Our goal in this appendix is to define a topology on $\mathcal O(\mathbb C^m,0)$ with the following property:

\medskip

\noindent{\bf Local uniform convergence.} \emph{
A sequence of germs $(f_k)$ converges to a germ $f$ if there exists $r>0$ such that $f$ and all the $f_n$ admit a representative which is bounded and holomorphic on $\mathbb B_r$ and $$\sup_{\Vert z\Vert\le r}\Vert f_k(z)-f(z)\Vert\to 0\quad \textrm{as}\quad n\to\infty.$$
}

\medskip

The construction of such a topology, given in Theorem \ref{convergentsequences} below, resembles very closely the so-called inductive limit topology of a sequence of nested Fr\'echet spaces $X_1\subset X_2\subset \dots\subset X_n \subset \dots$. See \cite{Ru} or \cite{Gr} for an example of this construction in the case of $C^\infty_0(\Omega)$. For a discussion of inductive limit topology for holomorphic germs, see for instance \cite{M}, where germs in possibly infinite dimensional spaces are considered. While the results presented here are undoubtedly known to experts on the topic, we include them for the sake of completeness.

Suppose that $\varepsilon_n\to 0$ is a strictly decreasing sequence. Let $X_n$ be the space of all bounded and holomorphic functions on $\mathbb B_{\varepsilon_n}$ equipped with the distance
\[
d_n(f,g):=\sup_{\Vert z\Vert \le \varepsilon_n}\Vert f(z)-g(z)\Vert.
\]
There exists a natural injection of $X_n$ into $\mathcal O(\mathbb C^m,0)$, which we will denote by $\pi_n$. Furthermore it is clear that $X_n\subset X_{n+1}$.

We write $\tau_n$ for the standard metric topology on $X_n$. We note that $\tau_n$ is not equivalent to the compact-open topology. However, it will be clear that the inductive limit topology on $\mathcal{O}(\mathbb C^m,0)$ obtained using either sequence of topologies is identical. 

\begin{definition}[Inductive limit topology]
The \textit{inductive limit topology}, which we will denote by $\tau_{ind}$, is the finest topology on $\mathcal O(\mathbb C^m,0)$ such that each injection $\pi_n:X_n\rightarrow \mathcal O(\mathbb C^m,0)$ is a continuous function.

In other words a set $U$ is open in the inductive limit topology if and only if, for each $n$, $X_n\cap \pi_n^{-1}(U)$ is open with respect to the topology $\tau_n$.
\end{definition}

Since each map $\pi_n$ is injective, from now on we will avoid  writing the map $\pi_n$ and we will consider each $X_n$ as a subset of $\mathcal O(\mathbb C^m,0)$.

\begin{remark}
In the classical construction of the inductive limit topology of a nested sequence of Fr\'echet spaces, $X_1\subset X_2\subset\dots$, it is assumed that the topology on $X_n$ is induced by the topology on $X_{n+1}$, which fails in our setting. Consider for example the germs $f_n=n\left(\frac{z}{\varepsilon_1}\right)^n$. This sequence converges uniformly to $0$ on $\mathbb B_r$ for $r<\varepsilon_1$, in particular we have $d_2(f_n,0)\to 0$. On the other hand we see that $d_1(f_n,0)\to \infty$. This shows that $\tau_1$ is different from $\tau_2|_{X_1}$.
\end{remark}
\begin{lemma}
\label{derivation}
Let $\alpha\in\mathbb N^m$. The derivation $D_\alpha: f\mapsto \partial_\alpha f(0)$ is a continuous map from $\mathcal O(\mathbb C^m,0)$, endowed with the inductive limit topology, to $\mathbb C$.
\end{lemma}
\begin{proof}
If $D_\alpha$ was not continuous, there would exists $U\subset\mathbb C$ open and a natural number $n$ such that $V:=D_\alpha^{-1}(U)\cap X_n$ is not open with respect to the metric topology of $X_n$. If so, there exists $(f_k)\subset X_n\setminus V$ such that $f_k\to g\in V$ with respect to the metric $d_n$. By Weierstrass we have $D_\alpha f_k\to D_\alpha g$, which is not possible, since $f_k\not\in V$.
\end{proof}
\begin{corollary}
The inductive limit topology on $\mathcal O(\mathbb C^m,0)$ is Hausdorff.
\end{corollary}
\begin{proof}
Given two germs $f_1\neq f_2$ then there exists $\alpha\in \mathbb N^m$ such that $D_\alpha f\neq D_\alpha f_2$. The Hausdorff property of $\tau_{ind}$ follows from the continuity of $D_\alpha$.
\end{proof}

Recall that a subset $U$ of a topological space $X$ is \textit{sequentially open} if each sequence $(x_k)$ in $X$ converging to a point of $U$ is eventually in $U$. We say that $X$ is a \textit{sequential space} if every sequentially open subset of $X$ is open.
\begin{lemma}
\label{sequentialsp}$\mathcal O(\mathbb C^m,0)$, endowed with the inductive limit topology, is a sequential space.
\end{lemma}
\begin{proof}
First of all we notice that if $f_k\to f$ with respect to the metric topology of $X_n$, then the convergence is valid in $\tau_{ind}$. If this was not the case, we could find an open set $U\ni f$ in the inductive limit topology such that $(f_k)$ is not eventually contained in $U$. However this is not possible since $U\cap X_n$ is open, with respect to the metric topology of $X_n$, and $f_k$ converges to $f$ in this topology.

As a consequence, given $U\subset \mathcal O(\mathbb C^m,0)$ sequentially open, the set $U\cap X_n$ is sequentially open in the metric topology of $X_n$. Metric spaces are sequential, thus $U\cap X_n$ is open in $X_n$. By the definition of $\tau_{ind}$ it follows that $U$ is open, which concludes the proof.
\end{proof}

A proof of the following proposition can be found in \cite[pg. 209]{ETop}.
\begin{proposition}
Sequential compactness and countable compactness are equivalent in the class of sequential Hausdorff spaces.
\end{proposition}

Given $r>0$ and $f\in X_n$, we write $B_n(f,r)$ (respectively $\overline{B}_n(f,r)$) for the open (respectively closed) ball of radius $r$ and center $f$ with respect to the metric $d_n$.
\begin{lemma}
\label{relatcomp}
The closed ball $ \overline B_n(f,r)$ is compact in $(X_{n+1},d_{n+1})$.
\end{lemma}
\begin{proof}
It is sufficient to prove the lemma for the case $f=0$. Suppose $(f_k)$ is a sequence in $ \overline B_n(0,r)$, then by the weak Montel Theorem there exists a subsequence $k_m$ such that $f_{k_m}\to f_\infty$ uniformly on every compact subset $K\subset \mathbb B_{\varepsilon_n}$. In particular we have $d_{n+1}(f_{k_m},f_\infty)\to 0$.

The function $f_{\infty}$ is holomorphic on $\mathbb B_{\varepsilon_n}$. Furthermore, by uniform convergence on compact subsets of $\mathbb B_{\varepsilon_n}$, we have $\Vert f_\infty(z)\Vert\le r$ for every $z\in B_{\varepsilon_n}(0)$. This proves that $f_{\infty}\in \overline B_n(0,r)$, which concludes the proof of the lemma.
\end{proof}

\begin{theorem}
\label{convergentsequences}
A sequence $(f_k)$ is a convergent sequence with respect to $\tau_{ind}$ if and only if there exists $N$ such that $(f_k)\subset X_N$ and $(f_k)$ is convergent with respect to the metric $d_N$.
\end{theorem}
\begin{proof}
We already proved the only if part in Lemma \ref{sequentialsp}. Suppose now that $(f_k)$ is a convergent sequence in the inductive limit topology.
Write $f_\infty=\lim_{k\to\infty}f_k$ and let $n_0$ so that $f_\infty\in X_{n_0}$. Firstly we prove that there exists $N_0$ such that $(f_k)\subset X_{N_0}$. If this was not the case, by taking a subsequence if necessary, we may assume that for every $n$, $X_n\cap(f_k)$ is a finite set, and that the intersection is empty when $n\le n_0$.

Let $U=\mathcal O(\mathbb C^m,0)\setminus(f_k)$. The set $X_n\setminus U$ is finite for every $n\in\mathbb N$, therefore $U$ is open in the inductive limit topology.  It is clear that $f_\infty\in U$  but that the sequence $(f_k)$ is not eventually contained in $U$, contradicting $f_\infty=\lim_{k\to\infty}f_k$. The existence of $N_0$ follows.

Suppose now that for every $N\ge N_0$ we have $d_N(f_k,f_\infty)\not\to 0$. If $\varepsilon>0$ is fixed, then the sequence $f_k$ is not eventually contained in $B_N(f_\infty,\varepsilon)$. This follows from the fact that $B_N(f_\infty,\varepsilon)$ are relatively compact in $X_{N+1}$, plus the fact that convergence in $X_{N+1}$ implies convergence with respect to $\tau_{ind}$.

Let $(k_{n}^N)\subset \mathbb N$ be the subsequence obtained by removing all the indices $k$ for which $f_k\in B_N(f_\infty,\varepsilon)$.
Since $B_N(f_\infty,\varepsilon)\subset B_{N+1}(f_\infty,\varepsilon)$ it follows that
\[
(k_{n}^{N+1})\subset (k_{n}^N)\subset\dots\subset (k_{n}^{N_0}).
\]
Let $k_n:=k^{n+N_0}_n$. Then, for every $n\ge N-N_0$, we have that $f_{k_n}\not\in B_N(f_\infty,\varepsilon)$. Furthermore we notice that
\[
d_N(f_{k_n},f_\infty)\to\infty,\qquad 	\forall N\ge N_0.
\]
Suppose on the contrary that there exists $(k'_n)\subset (k_n)$ such that $d_N(f_{k'_n},f_\infty)\to M$. By the previous lemma, by taking a subsequence of $k'_n$ if necessary, we may assume that $f_{k'_n}\to g$ in $X_{N+1}$ with $g\neq f_\infty$.  Convergence in $X_{N+1}$ implies convergence in $\tau_{ind}$, therefore this is not possible.

The set $U=\mathcal O(\mathbb C^m,0)\setminus (f_{k_n})$ contains $f_\infty\in U$. Since $(f_k)\subset X_{N_0}$, given $N\ge N_0$ we have that $U\cap X_N=X_N\setminus(f_{k_n})$. The sequence $(f_{k_n})$ is divergent in $X_N$, thus $U\cap X_N$ is open in $X_N$ for $N\ge N_0$. Finally if $V$ is open in $X_{N+1}$ then $V\cap X_N$ is open in $X_N$. This proves that $U$ is an open neighborhood of $f_{\infty}$ in the inductive limit topology, which gives a contradiction. It follows that there exists $N_1$ such that $d_{N_1}(f_k,f_\infty)\to 0$.
\end{proof}
\begin{corollary}
\label{compactset}
A set $K\subset \mathcal O(\mathbb C^m,0)$ is sequentially compact in $\tau_{ind}$ if and only if there exists $N$ such that $K\subset X_N$ and is compact with respect to the metric topology of $X_N$.
\end{corollary}
\begin{proof}
Suppose that $K$ is sequentially compact in $\tau_{ind}$. First of all there exists $N_0$ such that $K\subset X_{N_0}$. If this was not the case, we could find a sequence $(f_k)\subset K$ such that $f_k\not\in X_k$. By the previous theorem, such sequence does not have any convergent subsequence with respect to $\tau_{ind}$, which contradicts sequential compactness.

Suppose that $K$ is unbounded in $X_N$ for every $k\ge N_0$. We can find $f_k\in K$ such that $d_k(f_k,0)> N$. The previous theorem proves that this sequence does not have convergent subsequence, which contradicts sequential compactness. This proves that $K$ is bounded in $X_{N_1}$ for some $N_1\ge N_0$. By Lemma \ref{relatcomp} the set $K$ is relatively  compact in $X_{N_1+1}$. Since the inductive limit topology is sequential, it follows that $K$ is a closed set in $\tau_{ind}$. This shows that $K$ is a compact set in $X_{N_1+1}$.

On the other hand since convergence in $X_N$ implies convergence in $\tau_{ind}$, a compact set $K\subset X_N$ is sequentially compact in $\tau_{ind}$.
\end{proof}
\begin{corollary}
\label{compactset2}
Compactness and sequential compactness are equivalent in $\mathcal O(\mathbb C^m,0)$.
\end{corollary}
\begin{proof}
Since $\mathcal O(\mathbb C^m,0)$ is a Hausdorff sequential space, we only have to prove that every sequentially compact set is compact. By the previous corollary if $K$ is sequentially compact then there exists $N$ such that $K\subset X_N$ and $K$ is compact in $X_N$. Let $\{U_\alpha\}$ be an open cover of $K$. By the definition of the inductive limit topology it follows that $\{U_\alpha\cap X_N\}$ is an open cover of $K$ in $X_N$. Using the compactness of $K$ in $X_N$ we can extract a finite open subcover $\{U_1,\dots,U_n\}$. This proves that $K$ is compact with respect to the topology $\tau_{ind}$.
\end{proof}
The following corollary is an immediate consequence of the previous ones.
\begin{corollary}
\label{compactset3}
A set $K\subset \mathcal O(\mathbb C^m,0)$ is compact in $\tau_{ind}$ if and only if there exists $N$ such that $K\subset X_N$ is compact with respect to the metric topology of $X_N$.
\end{corollary}

We conclude this appendix by showing that the inductive limit topology is not metrizable. Notice that a similar proof also shows that this topology is not even first countable.
\begin{proposition}
The inductive limit topology is not metrizable.
\end{proposition}
\begin{proof}
Suppose on the contrary that there exists a metric $d_{ind}$ on $\mathcal O(\mathbb C^m,0)$ such that $\tau_{ind}$ coincides with the metric topology of $d_{ind}$. We will write $B_{ind}(f,r)$ for the open ball of center $f$ and radius $r$ with respect of this metric.

We note that for every $n$ we can construct a sequence $(f^n_k)_{k\in\mathbb N}\subset X_{n}\setminus X_{n-1}$, such that $\lim_{k\to\infty}d_n(f^n_k,0)=0$. Since convergence in $X_n$ implies convergence in $\tau_{ind}$, for every $n$ we can find $k_n$ such that
\[
f^n_{k_n}\in B_{ind}(0,1/n).
\]
Now let $g_n=f^n_{k_n}$. It is clear that $d_{ind}(g_n,0)\to 0$ as $n\to\infty$. By Theorem \ref{convergentsequences}, it follows that $(g_n)$ is contained in some $X_N$, which is not possible by the definition of the $f^n_k$-s. This contradicts the fact that $\tau_{ind}$ is metrizable.
\end{proof}
\printbibliography
\end{document}